\documentclass[9pt,oneside,english]{amsart}
\usepackage[T1]{fontenc}
\usepackage{geometry}
\usepackage{array}
\usepackage{mathtools}
\usepackage{bm}
\usepackage{multirow}
\usepackage{amstext}
\usepackage{amsthm}
\usepackage{amssymb}
\usepackage{mathdots}

\makeatletter

\providecommand{\tabularnewline}{\\}

\numberwithin{equation}{section}
\numberwithin{figure}{section}
\theoremstyle{plain}
\newtheorem{thm}{\protect\theoremname}
\theoremstyle{remark}
\newtheorem{rem}[thm]{\protect\remarkname}
\theoremstyle{plain}
\newtheorem{lem}[thm]{\protect\lemmaname}
\theoremstyle{definition}
\newtheorem{example}[thm]{\protect\examplename}
\theoremstyle{plain}
\newtheorem{conjecture}[thm]{\protect\conjecturename}
\theoremstyle{plain}
\newtheorem{prop}[thm]{\protect\propositionname}

\usepackage{listings}
\usepackage{shuffle}

\makeatother

\usepackage{babel}
\providecommand{\conjecturename}{Conjecture}
\providecommand{\examplename}{Example}
\providecommand{\lemmaname}{Lemma}
\providecommand{\propositionname}{Proposition}
\providecommand{\remarkname}{Remark}
\providecommand{\theoremname}{Theorem}

\begin{document}
\address[Minoru Hirose]{Institute for Advanced Research, Nagoya University, Furo-cho, Chikusa-ku, Nagoya, 464-8602, Japan}
\email{minoru.hirose@math.nagoya-u.ac.jp}

\title{Modular phenomena for regularized double zeta values}
\begin{abstract}
In this paper, we investigate linear relations among regularized motivic
iterated integrals on $\mathbb{P}^{1}\setminus\{0,1,\infty\}$ of
depth two, which we call regularized motivic double zeta values. Some
mysterious connections between motivic multiple zeta values and modular
forms are known, e.g. Gangl--Kaneko--Zagier relation for the totally
odd double zeta values and Ihara--Takao relation for the depth graded
motivic Lie algebra. In this paper, we investigate so-called non-admissible
cases and give many new Gangl--Kaneko--Zagier type and Ihara--Takao
type relations for regularized motivic double zeta values. Specifically,
we construct linear relations among a certain family of regularized
motivic double zeta values from odd period polynomials of modular
forms for the unique index two congruence subgroup of the full modular
group. This gives the first non trivial example of a construction
of the relations among multiple zeta values (or their analogues) from
modular forms for a congruence subgroup other than the ${\rm SL}_{2}(\mathbb{Z})$.
\end{abstract}

\author{Minoru Hirose}

\maketitle
\tableofcontents{}

\section{Introduction}

Multiple zeta values (MZVs) are real numbers defined by
\[
\zeta(k_{1},\dots,k_{d})\coloneqq\sum_{0<m_{1}<\cdots<m_{d}}\frac{1}{m_{1}^{k_{1}}\cdots m_{d}^{k_{d}}}
\]
where $k_{1},\dots,k_{d}$ are positive integers with $k_{d}>1$.
Here $d$ is called the depth and $k_{1}+\cdots+k_{d}$ is called
the weight, and multiple zeta values of depth two are called double
zeta values. They are periods in the sense of Kontsevich-Zagier, more
specifically periods of mixed Tate motives over $\mathbb{Z}$, and
are lifted to algebraic objects $\zeta^{\mathfrak{m}}(k_{1},\dots,k_{d})\in\mathcal{P}_{{\rm MT}(\mathbb{Z})}^{\mathfrak{m}}$
called motivic multiple zeta values (or motivic double zeta values
if $d=2$). Here $\mathcal{P}_{{\rm MT}(\mathbb{Z})}^{\mathfrak{m}}$
is the ring of motivic periods of ${\rm MT}(\mathbb{Z})$ and there
exist a canonical ring morphism ${\rm per}:\mathcal{P}_{{\rm MT}(\mathbb{Z})}^{\mathfrak{m}}\to\mathbb{C}$
called the period map with the property ${\rm per}(\zeta^{\mathfrak{m}}(k_{1},\dots,k_{d}))=\zeta(k_{1},\dots,k_{d})$.
We denote by $\mathcal{H}\subset\mathcal{P}_{{\rm MT}(\mathbb{Z})}^{\mathfrak{m}}$
the $\mathbb{Q}$-algebra generated by $1$ and all motivic multiple
zeta values. Then $\mathcal{A}\coloneqq\mathcal{H}/\zeta^{\mathfrak{m}}(2)\mathcal{H}$
has a structure of Hopf algebra, and $\mathcal{H}$ becomes a Hopf
$\mathcal{A}$-comodule. See Section \ref{subsec:Motivic-multiple-zeta}
for further details.

Multiple zeta values also appear in other areas of pure and applied
mathematics such as in the Kontsevich integrals of knots (\cite{Le_Murakami_95},
\cite{Le_Murakami_96}, \cite{Le_Murakami_96_2}) in topology, or
in the evaluation of scattering amplitudes (\cite{amplitude_2013_1},
\cite{amplitude_2013_2}) in mathematical physics. Therefore they
are already objects of significant interest and study.

Furthermore, there are many studies on the connection between (motivic)
multiple zeta values and modular forms. The first indications of a
connection between modular forms and MZV's occurred in the formulation
of the Broadhurst-Kreimer conjecture \cite{Broadhurst}, when it was
observed that the dimensions of the weight and depth graded multiple
zeta values involves the term $S(x)=\frac{x^{12}}{(1-x^{4})(1-x^{6})}$
which encodes the dimensions of ${\rm SL}_{2}(\mathbb{Z})$-cusp forms.
The depth 2 case of this conjecture is (essentially) the result of
\cite{GKZ}, which establishes the modular origin of $\zeta({\rm odd},\text{{\rm odd}})$
relations (See Theorem \ref{thm:GKZ_intro}, below). The dual viewpoint
is studied in \cite{Schneps_poisson}, \cite{Baumard_Schneps_period_polynomial_relations}
for depth $2$ and \cite{Brown_Depth_Graded}, \cite{Brown_zeta_elements},
\cite{Ma_Tasaka} for higher depths. Interpretation of conjectural
dimension in Broadhurst-Kreimer conjecture is treated in \cite{Brown_Depth_Graded},
\cite{Li_depth_structure}. Furthermore, depth $2$ and weight odd
aspects are studied in \cite{Zagier_2_3_2}, \cite{Ding_Ma_odd_weight},
and modular phenomena of another type for level 2 or level 3 modular
forms are investigated in \cite{Kaneko_Tasaka}, \cite{Ding_Ma_muN}.
Extensions to $q$-analogues or Eisenstein-series-analogue of multiple
zeta values are studied in \cite{Bachman_modified_double}, \cite{Tasaka_Hecke_eigenform}.

The purpose of this paper is to establish new modular phenomena for
regularized motivic iterated integrals on $\mathbb{P}^{1}\setminus\{0,1,\infty\}$,
which we call regularized motivic multiple zeta values. Regularized
motivic multiple zeta values are generalization of usual motivic multiple
zeta values, which are convergent integral cases of regularized motivic
multiple zeta values. Regularized motivic multiple zeta values are
very natural generalization of usual motivic multiple zeta values
and used in various situations\footnote{For example, regularized motivic multiple zeta values appear even
in the calculation of the coproduct of usual motivic multiple zeta
values. In addition, regularized motivic multiple zeta values appear
in the coefficient of motivic KZ-associator.}, however, there seems to be no work on their modular phenomena. In
this paper, we give a lot of modular phenomena for regularized motivic
double zeta values. Some of them are similar to the known cases of
usual motivic double zeta values while others are not.

We denote by $V_{w}^{\pm}$ the space of homogeneous polynomials $p(X,Y)$
of degree $w$ with rational coefficients such that $p(-X,Y)=\pm p(X,Y)$.
For an even $w$ and a modular form $f$ of weight $w+2$, an even
(resp. odd) period polynomial $P_{f}^{+}(X,Y)\in\mathbb{C}\otimes V_{w}^{+}$
(resp. $P_{f}^{-}(X,Y)\in\mathbb{C}\otimes V_{w}^{-}$) is defined
by using the special values of the L-function of $f$.

In this paper, we consider two types of modular phenomena:
\begin{description}
\item [{GKZ-type phenomena}] modular forms induce linear relations among
regularized motivic multiple zeta values,
\item [{Ihara--Takao-type phenomena}] modular forms induce algebraic relations
in the dual Hopf module of $\mathcal{H}$.
\end{description}

\subsection{Known GKZ-type phenomena for the usual motivic double zeta values}

The eponymous result in the area of GKZ-type phenomena is the following
well-known result due to Gangl-Kaneko-Zagier.
\begin{thm}[\cite{GKZ}, Gangl--Kaneko--Zagier]
\label{thm:GKZ_intro}Let $w$ be a nonnegative even integer, $k=w+2$,
and $f$ a modular form for ${\rm SL}_{2}(\mathbb{Z})$ of weight
$k$. Define $a_{0},\dots,a_{w}\in\mathbb{C}$ by
\[
\sum_{r=0}^{w}a_{r}x^{r}y^{w-r}=P_{f}^{+}(x+y,x).
\]
Then
\[
\sum_{\substack{r=0\\
r:{\rm even}
}
}^{w-2}a_{r}r!(w-r)!\zeta^{\mathfrak{m}}(r+1,w-r+1)\equiv0\pmod{\mathbb{C}\zeta^{\mathfrak{m}}(k)}.
\]
Furthermore, this gives a one-to-one correspondence between the modular
forms for ${\rm SL}_{2}(\mathbb{Z})$ of weight $k$ and the linear
relations among 
\[
\{\zeta^{\mathfrak{m}}(r+1,s+1)\mid r+s=w,r\geq0,s>0,r:{\rm even}\}
\]
modulo $\mathbb{C}\zeta^{\mathfrak{m}}(k)$.
\end{thm}

\begin{rem}
In fact, a more refined equality in the space of formal double zeta
value is proved in \cite{GKZ}. The space of formal double zeta values
is defined as the set of formal sums of indices modulo double shuffle
relations, and motivic multiple zeta values satisfy double shuffle
relations, thus Gangl--Kaneko--Zagier's original result implies
the former part of Theorem \ref{thm:GKZ_intro}. If we assume that
the relations among motivic double zeta values are exhausted by double
shuffle relations of depth two and ``Euler's relations'' $\zeta^{\mathfrak{m}}(a)\zeta^{\mathfrak{m}}(b)\in\mathbb{Q}\zeta^{\mathfrak{m}}(a+b)$
($a,b$: even) (which seems to be reasonable), the latter part of
Theorem \ref{thm:GKZ_intro} is also a consequence of their result.
\end{rem}

\begin{rem}
In Theorem \ref{thm:GKZ_intro}, the modulus is taken as $\mathbb{C}\zeta^{\mathfrak{m}}(k)$
rather than $\mathbb{Q}\zeta^{\mathfrak{m}}(k)$ so as to allow any
modular forms. If we start with a modular form whose even period polynomial
has rational coefficients, then Theorem \ref{thm:GKZ_intro} gives
$\mathbb{Q}$-linear relations among motivic MZVs. It is also known
that any modular form can be expressed as a $\mathbb{C}$-linear sum
of such modular forms.
\end{rem}

Theorem \ref{thm:GKZ_intro} gives a complete description of the linear
relations among motivic double zeta values of type $\zeta^{\mathfrak{m}}({\rm odd},{\rm odd})$
in terms of even period polynomials of modular forms. Thus this theorem
gives rise to the following questions: How about the linear relations
among $\zeta^{\mathfrak{m}}({\rm odd},{\rm even})$'s, $\zeta^{\mathfrak{m}}({\rm even},{\rm odd})$'s,
or $\zeta^{\mathfrak{m}}({\rm even},{\rm even})$'s. The answers to
these questions are also known as explained below.

For the first case, for an odd integer $k>2$, Ma constructed linear
relations among $\zeta^{\mathfrak{m}}({\rm odd},{\rm even})$'s of
weight $k$ from the odd (resp. even) period polynomial of modular
forms of weight $k-1$ (resp. $k+1$) for ${\rm SL}_{2}(\mathbb{Z})$,
and Li--Liu showed that all the linear relations among $\zeta^{\mathfrak{m}}({\rm odd},{\rm even})$'s
are essentially exhausted by Ma's relations. Hereafter, we use the
expression such as $\frac{\partial}{\partial x}p(ax+by,cx+dy)$ for
the partial derivative of $p(ax+by,cx+dy)$ with respect to $x$.
We remind the reader that it should not be confused with $\left(\frac{\partial p}{\partial x}\right)(ax+by,cx+dy)$
which is the partial derivative of $p(u,v)$ with respect to $u$
evaluated at $(u,v)=(ax+by,cx+dy)$.
\begin{thm}[{\cite[Theorems 2 and 3]{Ding_Ma_odd_weight}, Ma}]
\label{thm:Ma}Let $w$ be an odd positive integer, $k=w+2$, $f$
a cusp form for ${\rm SL}_{2}(\mathbb{Z})$ of weight $k-1$, and
$g$ a cusp form for ${\rm SL}_{2}(\mathbb{Z})$ of weight $k+1$.
Define $a_{0},\dots,a_{w}\in\mathbb{C}$ by
\[
\sum_{r=0}^{w}a_{r}X^{r}Y^{w-r}=XP_{f}^{-}(X+Y,X)+\frac{\partial}{\partial Y}P_{g}^{+}(X+Y,X).
\]
Then we have
\[
\sum_{\substack{r=0\\
r:{\rm even}
}
}^{w-1}(a_{r}-a_{w-r})r!(w-r)!\zeta^{\mathfrak{m}}(r+1,w-r+1)\equiv0\pmod{\mathbb{C}\zeta^{\mathfrak{m}}(k)}.
\]
Furthermore, these linear relations are linearly independent, i.e.,
if $a_{r}-a_{w-r}=0$ for all even $r$ then $f=g=0$.
\end{thm}

\begin{rem}
As in the case of Theorem \ref{thm:GKZ_intro}, we can obtain a congruence
modulo $\mathbb{Q}\zeta^{\mathfrak{m}}(k)$ if the coefficients of
$P_{f}^{-}(X,Y)$ and $P_{g}^{+}(X,Y)$ are both rational, and any
cusp form of weight $k-1$ (resp. $k+1$) is a $\mathbb{C}$-linear
sum of the modular forms $f$ (resp. $g$) whose odd (resp. even)
period polynomials are rational. Furthermore, Ma's original theorem
is stated as an equality in the formal double zeta space. Since all
linear relations among odd weight motivic double zeta values are given
by double shuffle relations, the result of Ma is equivalent to Theorem
\ref{thm:Ma}.
\end{rem}

\begin{thm}[\cite{Ding_Ma_odd_weight}, Ma]
\label{thm:Ma_trivial}Let $k\geq3$ be an odd integer. Then we have
\[
2\zeta^{\mathfrak{m}}(k-2,2)-(k-2)\zeta^{\mathfrak{m}}(1,k-1)+\sum_{\substack{r+s=k\\
1\leq r\leq k-2:{\rm odd}
}
}(r-s)\zeta^{\mathfrak{m}}(r,s)\equiv0\pmod{\mathbb{Q}\zeta^{\mathfrak{m}}(k)}.
\]
\end{thm}

\begin{rem}
Theorem \ref{thm:Ma} can be extend to the case where $f$ is an Eisenstein
series for ${\rm SL}_{2}(\mathbb{Z})$. One may guess that Theorem
\ref{thm:Ma_trivial} is a multiple of the case $(f,g)=(E_{k-1},0)$
of Theorem \ref{thm:Ma}, but this is \emph{not} correct in general.
\end{rem}

\begin{thm}[\cite{Li_Liu}, Li--Liu]
\label{thm:Li-Liu}Ma's relations in Theorems \ref{thm:Ma} and \ref{thm:Ma_trivial}
exhaust all the linear relations among $\zeta^{\mathfrak{m}}({\rm odd},{\rm even})$'s
modulo $\mathbb{Q}\zeta^{\mathfrak{m}}(k)$.
\end{thm}

For the second case, Zagier showed the following result.
\begin{thm}[{\cite[Theorem 2, Lemma 3]{Zagier_2_3_2}, Zagier}]
\label{thm:Zagier}There are no linear relations among $\zeta^{\mathfrak{m}}({\rm even},{\rm odd})$'s
and $\zeta^{\mathfrak{m}}({\rm odd})$'s.
\end{thm}

\begin{rem}
The original statement in \cite[Theorem 2]{Zagier_2_3_2} is slightly
different from Theorem \ref{thm:Zagier} since a motivic setting is
not introduced in \cite{Zagier_2_3_2}. However, we can easily extract
a proof of Theorem \ref{thm:Zagier} from Zagier's original proof.
More precisely, for an odd number $k=2K+1\geq3$, Zagier showed that
\[
\zeta(2r,k-2r)\equiv\sum_{s=1}^{K-1}\left({2K-2s \choose 2r-1}+{2K-2s \choose 2K-2r}\right)\zeta(2s)\zeta(k-2s)\pmod{\mathbb{Q}\zeta(k)}
\]
\cite[(36)]{Zagier_2_3_2} and
\[
\det\left({2K-2s \choose 2r-1}+{2K-2s \choose 2K-2r}\right)_{1\leq r,s\leq K-1}\neq0
\]
\cite[Lemma 3]{Zagier_2_3_2}. By (\ref{eq:coprod_lemma}), we can
lift the first congruence to the one for motivic multiple zeta values.
Here $\zeta^{\mathfrak{m}}(2s)\zeta^{\mathfrak{m}}(k-2s)$'s with
$1\leq s\leq K-1$ and $\zeta^{\mathfrak{m}}(k)$ are linearly independent.
Thus the non-vanishing of the determinant implies Theorem \ref{thm:Zagier}.
\end{rem}

For the third case, the harmonic product formula gives $\zeta^{\mathfrak{m}}(2a,2b)+\zeta^{\mathfrak{m}}(2b,2a)\in\mathbb{Q}\zeta^{\mathfrak{m}}(2a+2b)$.
In fact, these relations exhaust all the linear relations among $\zeta^{\mathfrak{m}}({\rm even},{\rm even})$'s
modulo $\mathbb{Q}\zeta^{\mathfrak{m}}({\rm even})$.
\begin{thm}[Tasaka, personal communication]
\label{thm:Tasaka}The linear relations among $\zeta^{\mathfrak{m}}({\rm even},{\rm even})$'s
modulo $\mathbb{Q}\zeta^{\mathfrak{m}}({\rm even})$ are exhausted
by the relations $\zeta^{\mathfrak{m}}(2a,2b)+\zeta^{\mathfrak{m}}(2b,2a)\equiv0\pmod{\mathbb{Q}\zeta^{\mathfrak{m}}(2a+2b)}$.
\end{thm}

\begin{rem}
Let $\Gamma_{B}\subset{\rm SL}_{2}(\mathbb{Z})$ be the congruence
subgroup defined in Section \ref{subsec:modularforms}. Then the space
of odd symmetric polynomials of degree $w$ can be interpreted as
the space of odd period polynomials for $\Gamma_{B}$ (Proposition
\ref{prop:period_polynomials}). Thus Theorem \ref{thm:Tasaka} says
that
\[
\sum_{\substack{0<r<w\\
r:{\rm odd}
}
}a_{r}r!(w-r)!\zeta^{\mathfrak{m}}(r+1,w-r+1)\equiv0\pmod{\mathbb{C}\zeta^{\mathfrak{m}}(w+2)}
\]
if and only if there exists a modular form $f$ for $\Gamma_{B}$
of weight $w+2$ such that $P_{f}^{+}(X,Y)=\sum_{\substack{0<r<w\\
r:{\rm odd}
}
}a_{r}X^{r}Y^{w-r}$.
\end{rem}

\begin{rem}
The author learned Theorems \ref{thm:Tasaka}, \ref{thm:Tasaka-dual}
and their proofs from Tasaka. We give a proof of Theorems \ref{thm:Tasaka}
and \ref{thm:Tasaka-dual} in Section \ref{sec:0_odd_odd}.
\end{rem}

\subsection{Known Ihara--Takao-type phenomena for usual motivic double zeta
values}

Since $\mathcal{A}$ (resp. $\mathcal{H}$) has a structure of a Hopf
algebra (resp. $\mathcal{A}$-comodule), the dual vector space $\mathcal{A}^{\vee}$
(resp. $\mathcal{H}^{\vee}$) of $\mathcal{A}$ (resp. $\mathcal{H}$)
has a structure of non-commutative algebra (resp. $\mathcal{A}^{\vee}$-module):
\[
\begin{split}\mathcal{A}^{\vee}\times\mathcal{A}^{\vee}\to\mathcal{A}^{\vee};\ (\sigma,\sigma')\mapsto\sigma\sigma',\\
\mathcal{A}^{\vee}\times\mathcal{H}^{\vee}\to\mathcal{H}^{\vee};\ (\sigma,\sigma')\mapsto\sigma\sigma'.
\end{split}
\]
For an odd (resp. even) integer $n\geq2$, let $\sigma_{n}$ be any
element of $\mathcal{A}^{\vee}$ (resp. $\mathcal{H}^{\vee}$) such
that $\left\langle \sigma_{n},1\right\rangle =0$ and $\left\langle \sigma_{n},\zeta^{\mathfrak{m}}(m)\right\rangle =\delta_{n,m}$.
We denote by $\bar{H}_{w}$ the space spanned by all motivic multiple
zeta values of depth $2$ and weight $w+2$. Then for any $n$ and
$m$, $\left.\sigma_{n}\sigma_{m}\right|_{\bar{H}_{w}}$ does not
depend on the choice of $\sigma_{n}\in\mathcal{A}^{\vee}$ and $\sigma_{m}\in\mathcal{H}^{\vee}$.
\begin{thm}[\cite{Ihara-Takao}, \cite{Schneps_poisson}, Ihara--Takao, Schneps]
\label{thm:Ihara-Takao}Let $w$ be a nonnegative even integer, $k=w+2$,
and $a_{2},a_{4},\dots,a_{w-2}$ be complex numbers. Then
\[
\left.\sum_{2\leq r\leq w-2,r:{\rm even}}a_{r}\sigma_{r+1}\sigma_{w-r+1}\right|_{\bar{H}_{w}}=0
\]
if and only if there exists $f\in M_{k}({\rm SL}_{2}(\mathbb{Z}))$
such that
\[
P_{f}^{+}(X,Y)=\sum_{2\leq r\leq w-2,r:{\rm even}}a_{r}X^{r}Y^{w-r}.
\]
\end{thm}

On the other hand, for odd $w$, $\left.\sigma_{n}\sigma_{m}\right|_{\bar{H}_{w}}$'s
are linearly independent. For subsets $N,N'\subset\mathbb{Z}_{\geq0}$,
we denote by $\bar{H}_{w}(N,N')$ the subspace of $\bar{H}_{w}$ spanned
by
\[
\{\zeta^{\mathfrak{m}}(r+1,s+1):r+s=w,s>0,r\in N,s\in N'\}.
\]
We denote by ${\bf odd}$ (resp. ${\bf even}$) the set of nonnegative
odd (resp. even) integers. Then it is known that $\bar{H}_{w}({\bf even},{\bf even})=\bar{H}_{w}$
for an even $w$ (\cite[Theorem 2]{GKZ}) and $\bar{H}_{w}({\bf odd},{\bf even})=\bar{H}_{w}$
for an odd $w$ (\cite[Theorem 2]{Zagier_2_3_2}). However $\bar{H}_{w}({\bf even},{\bf odd})$
and $\bar{H}_{w}({\bf odd},{\bf odd})$ do not coincide with $\bar{H}_{w}$
in general. The following three theorems gives complete descriptions
of all the linear relations among $\left.\sigma_{n}\sigma_{m}\right|_{\bar{H}_{w}({\bf even},{\bf odd})}$'s
and those among $\left.\sigma_{n}\sigma_{m}\right|_{\bar{H}_{w}({\bf odd},{\bf odd})}$'s.
\begin{thm}[{\cite[Section 6]{Zagier_2_3_2}, Zagier}]
\label{thm:Zagier_dual_relation}Let $w$ be an odd positive integer,
$w=k+2$, $f$ a cusp form for ${\rm SL}_{2}(\mathbb{Z})$ of weight
$k-1$, and $g$ a cusp form for ${\rm SL}_{2}(\mathbb{Z})$ of weight
$k+1$. Define $a_{0},\dots,a_{w}\in\mathbb{C}$ by
\[
\sum_{r=0}^{w}a_{r}X^{r}Y^{w-r}=XP_{f}^{-}(X,Y)+\frac{\partial}{\partial Y}P_{g}^{+}(X,Y).
\]
Then we have
\[
\left.\sum_{2\leq r\leq w-1,r:{\rm even}}a_{r}\sigma_{r+1}\sigma_{w-r+1}\right|_{\bar{H}_{w}({\bf even},{\bf odd})}=0.
\]
Furthermore, these linear relations are linearly independent, i.e.,
if $a_{r}=0$ for all even $r$ then $f=g=0$.
\end{thm}

\begin{thm}[\cite{Li_Liu}, Li--Liu]
\label{thm:Li-Liu-dual}Zagier's relations in Theorem \ref{thm:Zagier_dual_relation}
exhaust all linear relations among $\left.\sigma_{a}\sigma_{b}\right|_{\bar{H}_{w}({\bf even},{\bf odd})}$'s.
\end{thm}

\begin{thm}[Tasaka, personal communication]
\label{thm:Tasaka-dual}Let $w$ be a nonnegative even integer and
$k=w+2$. Then
\[
\left.\sum_{0<r<w,r:{\rm even}}a_{r}\sigma_{r+1}\sigma_{w-r+1}\right|_{\bar{H}_{w}}=0
\]
if and only if there exists a weight $k$ modular forms $f$ for $\Gamma_{0}(2)$
such that
\[
P_{f}^{+}(X,Y)=\sum_{0<r<w,r:{\rm even}}a_{w-r}X^{r}Y^{w-r}.
\]
\end{thm}

\subsection{The first main theorem: GKZ-type modular phenomena for non admissible
motivic double zeta values}

Recall that a motivic multiple zeta value $\zeta^{\mathfrak{m}}(k_{1},\dots,k_{d})$
for $k_{d}>1$ is defined by a motivic admissible iterated integral
\[
(-1)^{d}I^{\mathfrak{m}}(0;10^{k_{1}-1}\cdots10^{k_{d}-1};1).
\]
Let $0'$ and $1'$ be tangential base points at $0$ and $1$ respectively
such that $I^{\mathfrak{m}}(0';0;1')=0$ and $I^{\mathfrak{m}}(0';1;1')=-T$
where $T$ is an indeterminate. We define a regularized motivic multiple
zeta value by
\[
J(k_{0};k_{1},\dots,k_{d}):=I^{\mathfrak{m}}(0';0^{k_{0}}10^{k_{1}}\cdots10^{k_{d}};1')\in\mathcal{H}[T].
\]
We call $J(k_{0};k_{1},\dots,k_{d})$ an admissible (resp. non admissible)
motivic multiple zeta value if $k_{0}=0$ and $k_{d}>0$ (resp. $k_{0}\neq0$
or $k_{d}=0$). By definition, an admissible motivic multiple zeta
values is a usual motivic multiple zeta value, i.e., 
\[
J(0;k_{1},\dots,k_{d})=(-1)^{d}\zeta^{\mathfrak{m}}(k_{1}+1,\dots,k_{d}+1).
\]
Note that an explicit expression of a regularized motivic multiple
zeta value in terms of admissible motivic multiple zeta values is
given by
\[
J(k_{0};k_{1},\dots,k_{d})=(-1)^{k_{0}+d}\sum_{\substack{l_{1}+\cdots+l_{d}=k_{0}\\
l_{1},\dots,l_{d}\geq0
}
}{k_{1}+l_{1} \choose l_{1}}\cdots{k_{d}+l_{d} \choose l_{d}}\zeta^{\mathfrak{m},\shuffle}(k_{1}+l_{1}+1,\dots,k_{d}+l_{d}+1;T)
\]
where $\zeta^{\mathfrak{m},\shuffle}(k_{1}+l_{1}+1,\dots,k_{d}+l_{d}+1;T)$
is the shuffle regularized polynomial (e.g. $\zeta^{\mathfrak{m}}(1;T)=T$).
Note that Theorem \ref{thm:GKZ_intro} is concerned with $J(0;{\rm even},{\rm even})$'s,
Theorems \ref{thm:Ma}, \ref{thm:Ma_trivial}, \ref{thm:Li-Liu},
\ref{thm:Zagier_dual_relation} and \ref{thm:Li-Liu-dual} are concerned
with $J(0;{\rm even},{\rm odd})$'s, Theorem \ref{thm:Zagier} is
concerned with $J(0;{\rm odd},{\rm even})$'s, and Theorems \ref{thm:Tasaka}
and \ref{thm:Tasaka-dual} are concerned with $J(0;{\rm odd},{\rm odd})$'s,
and Theorem \ref{thm:Ihara-Takao} is concerned with the space of
all motivic double zeta values (which is equal to the space spanned
by $J(0;{\rm even},{\rm even})$'s). Thus the following 8 questions
about the behavior of $J(N;0;N')$'s and $J(N;N';0)$'s for $N,N'\in\{{\rm even},{\rm odd}\}$
then naturally arise. In this paper, we answer all these questions.
Furthermore, we also give complete answers to the same questions about
$J(1;{\rm even},{\rm even})$'s, $J(1;{\rm even},{\rm odd})$'s, $J({\rm even};1,{\rm even})$'s,
$J({\rm even};1,{\rm odd})$'s and $J({\rm odd};1,{\rm odd})$'s.

We put $\tilde{\mathcal{H}}=\mathcal{H}[T]$. We define the depth
filtration $\mathfrak{D}$ on $\tilde{\mathcal{H}}$ by
\[
\mathfrak{D}_{d}\tilde{\mathcal{H}}=\left\langle \zeta^{\mathfrak{m},\shuffle}(n_{1},\dots,n_{i};T)\Bigg|i\leq d,(n_{1},\dots,n_{i})\in\mathbb{Z}_{\geq1}^{i}\right\rangle _{\mathbb{Q}}.
\]
We define the depth graded $J$-value $J_{\mathfrak{D}}(k_{0};k_{1},\dots,k_{d})$
to be the image of $J(k_{0};k_{1},\dots,k_{d})$ in ${\rm gr}_{d}^{\mathfrak{D}}\tilde{\mathcal{H}}:=\mathfrak{D}_{d}\tilde{\mathcal{H}}/\mathfrak{D}_{d-1}\tilde{\mathcal{H}}$.
We put $\bm{0}=\{0\}$ and $\bm{1}=\{1\}$. For subsets $N,N',N''\subset\mathbb{Z}_{\geq0}$,
we denote by $H_{w}(N;N',N'')$ the subspace of $\mathfrak{D}_{d}\tilde{\mathcal{H}}$
spanned by
\[
\{J_{\mathfrak{D}}(r;s,t):r+s+t=w,r\in N,s\in N',t\in N''\}.
\]
We put $H_{w}:=H_{w}(\mathbb{Z}_{\geq0},\mathbb{Z}_{\geq0},\mathbb{Z}_{\geq0})$.
We denote by $W_{w}^{\Gamma,+}\subset V_{w}^{+}$ (resp. $W_{w}^{\Gamma,-}\subset V_{w}^{-}$)
the space of even (resp. odd) period polynomials with rational coefficients
for weight $w+2$ modular (resp. cusp) forms for $\Gamma$ (see Section
\ref{subsec:modularforms} for detail). We put $W_{w}^{\pm}:=W_{w}^{\pm,{\rm SL}_{2}(\mathbb{Z})}$.
Then Theorem \ref{thm:GKZ_intro} can be restated as the existence
of the exact sequence

\begin{align}
 & 0\to &  & W_{w}^{+} &  & \xrightarrow{u_{{\rm GKZ}}} &  & V_{w}^{+} &  & \xrightarrow{X^{a}Y^{b}\mapsto a!b!J_{\mathfrak{D}}(0;a,b)} &  & H_{w}({\bf 0};{\bf even},{\bf even}) &  & \to0 &  & (w:{\rm even})\label{eq:exact_0_even_even}
\end{align}
where $u_{{\rm GKZ}}$ is a linear map defined by
\[
u_{{\rm GKZ}}(p(X,Y)):=\frac{1}{2}\left(p(-X-Y,X)+p(X-Y,X)\right).
\]
Since the left-hand side of Theorem \ref{thm:Ma_trivial} is equal
to
\begin{align*}
 & 2\zeta^{\mathfrak{m}}(k-2,2)-(k-2)\zeta^{\mathfrak{m}}(1,k-1)+\sum_{\substack{r+s=k\\
1\leq r\leq k-2:{\rm odd}
}
}(r-s)\zeta^{\mathfrak{m}}(r,s)\\
 & =2J(0;w-1,1)-wJ(0;0,w)+\sum_{\substack{r+s=w\\
0\leq r\leq w-1:{\rm even}
}
}(r-s)J(0;r,s)\ \ \ (w\coloneqq k-2,r\coloneqq r-1,s\coloneqq s-1)\\
 & =\Psi\left(\frac{2X^{w-1}Y}{(w-1)!}-\frac{Y^{w}}{(w-1)!}+X\sum_{\substack{r+s=w\\
0<r\leq w-1:{\rm even}
}
}\frac{X^{r-1}Y^{s}}{(r-1)!s!}-Y\sum_{\substack{r+s=w\\
0\leq r\leq w-1:{\rm even}
}
}\frac{X^{r}Y^{s-1}}{r!(s-1)!}\right)\ \ \ (\Psi(X^{a}Y^{b})\coloneqq a!b!J(0;a,b))\\
 & =\frac{1}{(w-1)!}\Psi\left(2X^{w-1}Y-Y^{w}+X\frac{(X+Y)^{w-1}-(X-Y)^{w-1}}{2}-Y\frac{(X+Y)^{w-1}+(X-Y)^{w-1}}{2}\right)\\
 & =\frac{1}{2(w-1)!}\Psi\left((X-Y)(X+Y)^{w-1}-(X+Y)(X-Y)^{w-1}-2Y^{w}+4X^{w-1}Y\right),
\end{align*}
Theorems \ref{thm:Ma} and \ref{thm:Li-Liu} can be restated as the
existence of the following exact sequence:
\begin{align}
 & 0\to &  & W_{w-1}^{-}\times\tilde{W}_{w+1}^{+}\times\mathbb{Q} &  & \xrightarrow{u_{{\rm M}}} &  & V_{w}^{+} &  & \xrightarrow{X^{a}Y^{b}\mapsto a!b!J_{\mathfrak{D}}(0;a,b)} &  & H_{w}({\bf 0};{\bf even},{\bf odd}) &  & \to0 &  & (w:{\rm odd})\label{eq:exact_0_even_odd}
\end{align}
where
\[
\tilde{W}_{w+1}^{+}:=W_{w+1}^{+}/\mathbb{Q}(X^{w+1}-Y^{w+1})
\]
and $u_{{\rm M}}$ is a linear map defined by
\begin{align*}
u_{{\rm M}}(p(X,Y),q(X,Y),c) & =\frac{1}{2}\left(R(X,Y)-R(Y,X)+R(-X,Y)-R(Y,-X)\right)\\
 & \ +c\left((X-Y)(X+Y)^{w-1}-(X+Y)(X-Y)^{w-1}-2Y^{w}+4X^{w-1}Y\right)
\end{align*}
with
\[
R(X,Y):=Xp(X+Y,X)+\frac{\partial}{\partial Y}q(X+Y,X).
\]
Theorem \ref{thm:Zagier} is equivalent to the exactness of the sequence
\begin{align}
 &  &  & 0 &  & \to &  & V_{w}^{-} &  & \xrightarrow{X^{a}Y^{b}\mapsto a!b!J_{\mathfrak{D}}(0;a,b)} &  & H_{w}({\bf 0};{\bf odd},{\bf even}) &  & \to0 &  & (w:{\rm odd}),\label{eq:exact_0_odd_even}
\end{align}
and Theorem \ref{thm:Tasaka} can be restated as the existence of
the exact sequence

\begin{align}
 & 0\to &  & W_{w}^{-,\Gamma_{B}} &  & \xrightarrow{{\rm id}} &  & V_{w}^{-} &  & \xrightarrow{X^{a}Y^{b}\mapsto a!b!J_{\mathfrak{D}}(0;a,b)} &  & H_{w}({\bf 0};{\bf odd},{\bf odd}) &  & \to0 &  & (w:{\rm even})\label{eq:exact_0_odd_odd}
\end{align}
where ${\rm id}$ is just the inclusion map.

Let $\Gamma_{A}$ be the congruence subgroup of ${\rm SL}_{2}(\mathbb{Z})$
defined by
\[
\Gamma_{A}:=\left\{ \gamma\in{\rm SL}_{2}(\mathbb{Z})\Big|\gamma\equiv\left(\begin{array}{cc}
1 & 0\\
0 & 1
\end{array}\right),\left(\begin{array}{cc}
1 & 1\\
1 & 0
\end{array}\right),\left(\begin{array}{cc}
0 & 1\\
1 & 1
\end{array}\right)\pmod{2}\right\} .
\]
Our first main theorem is as follows:
\begin{thm}
\label{thm:first_main_result}There exist the following 13 exact sequences.
\begin{description}
\item [{Case 1}] 
\begin{align}
 & 0\to &  & W_{w}^{-,\Gamma_{A}} &  & \xrightarrow{{\rm id}} &  & V_{w}^{-} &  & \xrightarrow{X^{a}Y^{b}\mapsto a!b!J_{\mathfrak{D}}(a;b,0)} &  & H_{w}({\bf odd};{\bf odd},{\bf 0}) &  & \to0 &  & (w:{\rm even})\label{eq:exact_odd_odd_0}
\end{align}
\item [{Case 2}] 
\begin{align}
 & 0\to &  & W_{w}^{+} &  & \xrightarrow{{\rm id}} &  & V_{w}^{+} &  & \xrightarrow{X^{a}Y^{b}\mapsto a!b!J_{\mathfrak{D}}(a;b,0)} &  & H_{w}({\bf even};{\bf even},{\bf 0}) &  & \to0 &  & (w:{\rm even})\label{eq:exact_even_even_0}
\end{align}
\item [{Case 3}] 
\begin{align}
 & 0\to &  & W_{w}^{+} &  & \xrightarrow{u_{{\rm GKZ}}} &  & V_{w}^{+} &  & \xrightarrow{X^{a}Y^{b}\mapsto a!b!J_{\mathfrak{D}}(1;a,b)} &  & H_{w+1}({\bf 1};{\bf even},{\bf even}) &  & \to0 &  & (w:{\rm even})\label{eq:exact_1_even_even}\\
 & 0\to &  & W_{w}^{+} &  & \xrightarrow{u_{{\rm GKZ}}} &  & V_{w}^{+} &  & \xrightarrow{X^{a}Y^{b}\mapsto a!b!J_{\mathfrak{D}}(a;0,b)} &  & H_{w}({\bf even};{\bf 0},{\bf even}) &  & \to0 &  & (w:{\rm even})\label{eq:exact_even_0_even}\\
 & 0\to &  & W_{w}^{-} &  & \xrightarrow{u_{{\rm GKZ}}} &  & V_{w}^{-} &  & \xrightarrow{X^{a}Y^{b}\mapsto a!b!J_{\mathfrak{D}}(a;1,b)} &  & H_{w+1}({\bf odd};{\bf 1},{\bf odd}) &  & \to0 &  & (w:{\rm even})\label{eq:exact_odd_1_odd}
\end{align}
\item [{Case 4}] 
\begin{align}
 & 0\to &  & W_{w-1}^{-}\times\tilde{W}_{w+1}^{+}\times\mathbb{Q} &  & \xrightarrow{u_{{\rm M}}} &  & V_{w}^{+} &  & \xrightarrow{X^{a}Y^{b}\mapsto a!b!J_{\mathfrak{D}}(1;a,b)} &  & H_{w+1}({\bf 1};{\bf even},{\bf odd}) &  & \to0 &  & (w:{\rm odd})\label{eq:exact_1_even_odd}\\
 & 0\to &  & W_{w-1}^{-}\times\tilde{W}_{w+1}^{+}\times\mathbb{Q} &  & \xrightarrow{u_{{\rm M}}} &  & V_{w}^{+} &  & \xrightarrow{X^{a}Y^{b}\mapsto a!b!J_{\mathfrak{D}}(a;1,b)} &  & H_{w+1}({\bf even};{\bf 1},{\bf odd}) &  & \to0 &  & (w:{\rm odd})\label{eq:exact_even_1_odd}\\
 & 0\to &  & W_{w-1}^{+}\times W_{w+1}^{-} &  & \xrightarrow{u_{{\rm M}}'} &  & V_{w}^{-} &  & \xrightarrow{X^{a}Y^{b}\mapsto a!b!J_{\mathfrak{D}}(a;0,b)} &  & H_{w}({\bf odd};{\bf 0},{\bf even}) &  & \to0 &  & (w:{\rm odd})\label{eq:exact_odd_0_even}
\end{align}
where $u_{{\rm M}}'$ is a linear map defined by 
\[
u_{{\rm M}}'(p(X,Y),q(X,Y))=\frac{1}{2}\left(R(X,Y)+R(Y,X)-R(-X,Y)-R(Y,-X)\right)
\]
with $R(X,Y)=Xp(X+Y,X)+\frac{\partial}{\partial Y}q(X+Y,X)$.
\item [{Case 5}] 
\begin{align}
 & 0\to &  & W_{w}^{-,\Gamma_{B}} &  & \xrightarrow{{\rm id}} &  & V_{w}^{-} &  & \xrightarrow{X^{a}Y^{b}\mapsto a!b!J_{\mathfrak{D}}(a;0,b)} &  & H_{w}({\bf odd};{\bf 0},{\bf odd}) &  & \to0 &  & (w:{\rm even}).\label{eq:exact_odd_0_odd}
\end{align}
\item [{Case 6}] 
\begin{align}
 & 0\to &  & \mathbb{Q}A(X,Y) &  & \xrightarrow{{\rm id}} &  & V_{w}^{+} &  & \xrightarrow{X^{a}Y^{b}\mapsto a!b!J_{\mathfrak{D}}(a;0,b)} &  & H_{w}({\bf even};{\bf 0},{\bf odd}) &  & \to0 &  & (w:{\rm odd})\label{eq:exact_even_0_0dd}\\
 &  &  & 0 &  & \to &  & V_{w}^{+} &  & \xrightarrow{X^{a}Y^{b}\mapsto a!b!J_{\mathfrak{D}}(a;1,b)} &  & H_{w+1}({\bf even};{\bf 1},{\bf even}) &  & \to0 &  & (w:{\rm even})\label{eq:exact_even_1_even}\\
 &  &  & 0 &  & \to &  & V_{w}^{+} &  & \xrightarrow{X^{a}Y^{b}\mapsto a!b!J_{\mathfrak{D}}(a;b,0)} &  & H_{w}({\bf even};{\bf odd},{\bf 0}) &  & \to0 &  & (w:{\rm odd})\label{eq:exact_even_odd_0}\\
 &  &  & 0 &  & \to &  & V_{w}^{+} &  & \xrightarrow{X^{a}Y^{b}\mapsto a!b!J_{\mathfrak{D}}(b;a,0)} &  & H_{w}({\bf odd};{\bf even},{\bf 0}) &  & \to0 &  & (w:{\rm odd})\label{eq:exact_odd_even_0}
\end{align}
where
\[
A(X,Y)=(X+Y)^{w}+(-X+Y)^{w}.
\]
\end{description}
\end{thm}

Let us make some comments about this Theorem. Cases 1, 2 and 5 are
the cases where whose second arrows are identity maps, but the modular
groups differ. The most surprising case is Case 1, where the odd period
polynomials for $\Gamma_{A}$ appear. This seems to be the first non
trivial example\footnote{The reason why we did not count (\ref{eq:exact_0_odd_odd}) as the
first example is because the linear relation givens by (\ref{eq:exact_0_odd_odd})
are nothing but the simple relation $\zeta^{\mathfrak{m}}(2a,2b)+\zeta^{\mathfrak{m}}(2b,2a)\in\mathbb{Q}\zeta^{\mathfrak{m}}(2a+2b)$.
For the sake of uniform description, we related even such simple formulas
with the space of period polynomials, but the author does not know
whether this is a natural interpretation or not.} of constructing relations among multiple zeta values (or their analogues)
from modular forms for a congruence subgroup other than the ${\rm SL}_{2}(\mathbb{Z})$.
Cases 3, 4 and 5 deals with the exact sequences analogous to (\ref{eq:exact_0_even_even}),
(\ref{eq:exact_1_even_even}) and (\ref{eq:exact_0_odd_odd}), respectively.
Finally, Case 6 is (almost) no relation cases. More precisely, the
exactness of (\ref{eq:exact_even_0_0dd}) implies that there is exactly
one linear relation among $J_{\mathfrak{D}}({\rm even};0,{\rm odd})$'s,
whereas the exactness of (\ref{eq:exact_even_1_even}), (\ref{eq:exact_even_odd_0})
and (\ref{eq:exact_odd_even_0}) says that there are no linear relations
among $J_{\mathfrak{D}}({\rm even};1,{\rm even})$'s, $J_{\mathfrak{D}}({\rm even};{\rm odd},0)$'s
or $J_{\mathfrak{D}}({\rm odd};{\rm even},0)$'s.

\subsection{The second main theorem: Ihara--Takao-type phenomena for non admissible
motivic double zeta values}

For $n\geq0$, we put $J(n):=J(0;n)$ and $J_{\mathfrak{D}}(n):=J_{\mathfrak{D}}(0;n)$.
We extend the definition of $\sigma_{n}$. For a positive odd (resp.
even) integer $n$, let $\sigma_{n}$ be any element of $\tilde{A}^{\vee}$
(resp. $\mathcal{\tilde{H}}^{\vee}$) such that $\left\langle \sigma_{n},1\right\rangle =0$
and $\left\langle \sigma_{n},J(m-1)\right\rangle =\delta_{n,m}$.
For each $w$, we define a linear map $\lambda:V_{w}^{+}\to H_{w}^{\vee}$
by $\lambda(X^{i}Y^{j}):=\sigma_{i+1}\sigma_{j+1}$. Then, for any
submodule $N\subset H_{w}^{\vee}$, we can obtain the composite map
$V_{w}^{+}\xrightarrow{\lambda}H_{w}^{\vee}\to N^{\vee}$. By abuse
of notation, we also denote by $\lambda$ this composite map if there
is no risk of confusion. Theorem \ref{thm:Ihara-Takao} is equivalent
to the existence of the exact sequence

\begin{align}
 & 0\to &  & W_{w}^{+} &  & \xrightarrow{{\rm id}} &  & V_{w}^{+} &  & \xrightarrow{\lambda} &  & H_{w}^{\vee}\to0 &  & (w:{\rm even}),\label{eq:dual_even}
\end{align}
whereas, for odd $w$, it is easy to see that the following sequence
is exact:

\begin{align}
 &  &  & 0 &  & \to &  & V_{w}^{+} &  & \xrightarrow{\lambda} &  & H_{w}^{\vee}\to0 &  & (w:{\rm odd}).\label{eq:dual_odd}
\end{align}
Moreover, for even $w$, we know that $\bar{H}_{w}({\bf even},{\bf even})=\bar{H}_{w}$.
Hence (\ref{eq:dual_even}) can be restated as the existence of the
exact sequence
\begin{align}
 & 0\to &  & W_{w}^{+} &  & \xrightarrow{{\rm id}} &  & V_{w}^{+} &  & \xrightarrow{\lambda} &  & H_{w}({\bf 0};{\bf even};{\bf even})^{\vee}\to0 &  & (w:{\rm even}).\label{eq:dual_0_even_even}
\end{align}
Theorems \ref{thm:Zagier_dual_relation} and \ref{thm:Li-Liu-dual}
can be restated as the existence of the exact sequence

\begin{align}
 & 0\to &  & W_{w-1}^{-}\oplus W_{w+1}^{+} &  & \xrightarrow{(p,q)\mapsto Xp+\frac{\partial}{\partial Y}q} &  & V_{w}^{+} &  & \xrightarrow{\lambda} &  & H_{w}({\bf 0};{\bf even};{\bf odd})^{\vee}\to0 &  & (w:{\rm odd})\label{eq:dual_0_even_odd}
\end{align}
while Theorem \ref{thm:Tasaka-dual} can be restated as the existence
of the exact sequence
\begin{align}
 & 0\to &  & W_{w}^{+,\Gamma_{0}(2)} &  & \xrightarrow{p(X,Y)\mapsto p(Y,X)} &  & V_{w}^{+} &  & \xrightarrow{\lambda} &  & H_{w}({\bf 0};{\bf odd};{\bf odd})^{\vee}\to0 &  & (w:{\rm even}).\label{eq:dual_0_odd_odd}
\end{align}
For odd $w$, we know that $\bar{H}_{w}({\bf odd},{\bf even})=\bar{H}_{w}$.
Hence (\ref{eq:dual_odd}) can be restated as the existence of the
exact sequence

\begin{align}
 &  &  & 0 &  & \to &  & V_{w}^{+} &  & \xrightarrow{\lambda} &  & H_{w}({\bf 0};{\bf odd};{\bf even})^{\vee}\to0 &  & (w:{\rm odd}).\label{eq:dual_0_odd_even}
\end{align}
Our second main theorem is as follows:
\begin{thm}
\label{thm:second_main_result}There exist the following 13 exact
sequences (each case corresponds to the one in Theorem \ref{thm:first_main_result}).
\begin{description}
\item [{Case 1}] 
\begin{align}
 & 0\to &  & W_{w}^{-,\Gamma_{A}} &  & \xrightarrow{{\rm id}} &  & V_{w}^{+} &  & \xrightarrow{\lambda} &  & H_{w}({\bf odd};{\bf odd};{\bf 0})^{\vee}\to0 &  & (w:{\rm even},w\geq2)\label{eq:dual_odd_odd_0}
\end{align}
\item [{Case 2}] 
\begin{align}
 & 0\to &  & W_{w}^{+} &  & \xrightarrow{{\rm id}} &  & V_{w}^{+} &  & \xrightarrow{\lambda} &  & H_{w}({\bf even};{\bf even};{\bf 0})^{\vee}\to0 &  & (w:{\rm even})\label{eq:dual_even_even_0}
\end{align}
\item [{Case 3}] 
\begin{align}
 & 0\to &  & W_{w}^{+} &  & \xrightarrow{p\mapsto\int_{0}^{Y}pdY} &  & V_{w+1}^{+} &  & \xrightarrow{\lambda} &  & H_{w+1}({\bf 1};{\bf even};{\bf even})^{\vee}\to0 &  & (w:{\rm even})\label{eq:dual_1_even_even}\\
 & 0\to &  & W_{w}^{+} &  & \xrightarrow{{\rm id}} &  & V_{w}^{+} &  & \xrightarrow{\lambda} &  & H_{w}({\bf even};{\bf 0};{\bf even})^{\vee}\to0 &  & (w:{\rm even})\label{eq:dual_even_0_even}\\
 & 0\to &  & (\partial_{X}^{-1}W_{w}^{-}) &  & \xrightarrow{{\rm id}} &  & V_{w+1}^{+} &  & \xrightarrow{\lambda} &  & H_{w+1}({\bf odd};{\bf 1};{\bf odd})^{\vee}\to0 &  & (w:{\rm even})\label{eq:dual_odd_1_odd}
\end{align}
 where
\[
\partial_{X}^{-1}W_{w}^{-}:=\left\{ p\in V_{w+1}^{+}\Bigg|\frac{\partial}{\partial X}p(X,Y)\in W_{w}^{-}\right\} .
\]
\item [{Case 4}] 
\begin{align}
 & 0\to &  & (\partial_{Y}^{-1}W_{w-1}^{-})\times W_{w+1}^{+} &  & \xrightarrow{(p,q,c)\mapsto Xp+q} &  & V_{w+1}^{+} &  & \xrightarrow{\lambda} &  & H_{w+1}({\bf 1};{\bf even};{\bf odd})^{\vee}\to0 &  & (w:{\rm odd})\label{eq:dual_1_even_odd}\\
 & 0\to &  & (\partial_{X}^{-1}W_{w-1}^{-})\times W_{w+1}^{+} &  & \xrightarrow{(p,q,c)\mapsto Yp+q} &  & V_{w+1}^{+} &  & \xrightarrow{\lambda} &  & H_{w+1}({\bf even};{\bf 1};{\bf odd})^{\vee}\to0 &  & (w:{\rm odd})\label{eq:dual_even_1_odd}\\
 & 0\to &  & W_{w-1}^{+}\oplus W_{w+1}^{-} &  & \xrightarrow{(p,q)\mapsto Yp+\frac{\partial}{\partial X}q} &  & V_{w}^{+} &  & \xrightarrow{\lambda} &  & H_{w}({\bf odd};{\bf 0};{\bf even})^{\vee}\to0 &  & (w:{\rm odd})\label{eq:dual_odd_0_even}
\end{align}
where
\[
\partial_{Y}^{-1}W_{w-1}^{-}:=\left\{ p\in V_{w}^{-}\Bigg|\frac{\partial}{\partial Y}p(X,Y)\in W_{w-1}^{-}\right\} 
\]
and
\[
\partial_{X}^{-1}W_{w-1}^{-}:=\left\{ p\in V_{w}^{+}\Bigg|\frac{\partial}{\partial X}p(X,Y)\in W_{w-1}^{-}\right\} .
\]
\item [{Case 5}] 
\begin{align}
 & 0\to &  & W_{w}^{-,\Gamma_{0}(2)} &  & \xrightarrow{{\rm id}} &  & V_{w}^{+} &  & \xrightarrow{\lambda} &  & H_{w}({\bf odd};{\bf 0};{\bf odd})^{\vee}\to0 &  & (w:{\rm even},w\geq2)\label{eq:dual_odd_0_odd}
\end{align}
\item [{Case 6}] 
\begin{align}
 & 0\to &  & \mathbb{Q}Y^{w} &  & \xrightarrow{{\rm id}} &  & V_{w}^{+} &  & \xrightarrow{\lambda} &  & H_{w}({\bf even};{\bf 0};{\bf odd})^{\vee}\to0 &  & (w:{\rm odd})\label{eqdual_even_0_odd}\\
 &  &  & 0 &  & \to &  & V_{w+1}^{+} &  & \xrightarrow{\lambda} &  & H_{w+1}({\bf even};{\bf 1};{\bf even})^{\vee}\to0 &  & (w:{\rm even})\label{eq:dual_even_1_even}\\
 &  &  & 0 &  & \to &  & V_{w}^{+} &  & \xrightarrow{\lambda} &  & H_{w}({\bf even};{\bf odd};{\bf 0})^{\vee}\to0 &  & (w:{\rm odd})\label{eq:dual_even_odd_0}\\
 &  &  & 0 &  & \to &  & V_{w}^{+} &  & \xrightarrow{\lambda} &  & H_{w}({\bf odd};{\bf even};{\bf 0})^{\vee}\to0 &  & (w:{\rm odd})\label{eq:dual_odd_even_0}
\end{align}
\end{description}
\end{thm}

\subsection{Contents of the paper}

The paper is organized as follows. In Section \ref{sec:preliminaries},
we define some notions and give some lemmas. In Sections \ref{sec:gamma_a},
\ref{sec:strange_gkz}, \ref{sec:0_even_even}, \ref{sec:0_even_odd_DingMa}
and \ref{sec:0_odd_odd}, we discuss GKZ and Ihara--Takao type formulas
for the case $J({\rm odd};{\rm odd},0)$, the case $J({\rm even};{\rm even},0)$,
the cases similar to the case $J(0;{\rm even},{\rm even})$, the cases
similar to the case $J(0;{\rm even},{\rm odd})$, and the cases similar
to the case \ref{sec:0_odd_odd}, respectively. In Section \ref{sec:no-relation},
we discuss the cases $J({\rm even};1,{\rm even})$, $J({\rm even};{\rm odd},0)$
and $J({\rm odd};{\rm even},0)$, where there are almost no linear
relations. In Appendix \ref{sec:Period-polynomials}, we give explicit
expressions for the spaces of period polynomials for various congruence
subgroups that appear in this paper.

\section{Some preliminaries\label{sec:preliminaries}}

\subsection{\label{subsec:Motivic-multiple-zeta}Motivic multiple zeta values}

Let ${\rm MT}(\mathbb{Z})$ be the category of mixed Tate motives
over $\mathbb{Z}$ which is a Tannakian category extracted from Voevodsky's
triangulated category of motives (see \cite{DelGon05}). There exist
two special functors $\omega_{B},\omega_{dR}:{\rm MT}(\mathbb{Z})\to{\rm Vec}_{\mathbb{Q}}$
(Betti and de Rham realizations), and the ring of motivic periods
of ${\rm MT}(\mathbb{Z})$ is defined as the ring of functions on
the scheme of tensor isomorphism from $\omega_{dR}$ to $\omega_{B}$
\[
\mathcal{P}_{{\rm MT}(\mathbb{Z})}^{\mathfrak{m}}\coloneqq\mathcal{O}(\underline{{\rm Isom}}_{{\rm MT}(\mathbb{Z})}(\omega_{dR},\omega_{B}))
\]
and a special ring homomorphism ${\rm per}:\mathcal{P}_{{\rm MT}(\mathbb{Z})}^{\mathfrak{m}}\to\mathbb{C}$
called period map is naturally defined (see \cite{Brown_proceedings}).
For $x,y\in\{0,1\}$ and a word $w$ in $\{0,1\}$, the motivic iterated
integral $I^{\mathfrak{m}}(x;w;y)\in\mathcal{P}_{{\rm MT}(\mathbb{Z})}^{\mathfrak{m}}$
is defined. They satisfy $I^{\mathfrak{m}}(x;p_{1}\cdots p_{k};x)=\delta_{k,0}$,
$I^{\mathfrak{m}}(0;0;1)=I^{\mathfrak{m}}(0;1;1)=0$, the reversal
formula $I^{\mathfrak{m}}(x;p_{1}\cdots p_{k};y)=(-1)^{k}I^{\mathfrak{m}}(y;p_{k}\cdots p_{1};x)$,
and the shuffle product formula $I^{\mathfrak{m}}(x;w_{1}\shuffle w_{2};y)=I^{\mathfrak{m}}(x;w_{1};y)\cdot I^{\mathfrak{m}}(x;w_{2};y)$.
For an index $(k_{1},\dots,k_{d})\in\mathbb{Z}_{>0}^{d}$ with $d=0$
or $k_{d}>1$, a motivic multiple zeta value $\zeta^{\mathfrak{m}}(k_{1},\dots,k_{d})$
is defined by $(-1)^{d}I^{\mathfrak{m}}(0;10^{k_{1}-1}\cdots10^{k_{d}-1};1)$.
Furthermore, ${\rm per}(\zeta^{\mathfrak{m}}(k_{1},\dots,k_{d}))$
is equal to the multiple zeta value $\zeta(k_{1},\dots,k_{d})$. It
is shown by Brown (\cite{Brown_MixedTate}) that $\mathcal{P}_{{\rm MT}(\mathbb{Z})}^{\mathfrak{m}}$
is spanned by the motivic multiple zeta values and the reciprocal
of motivic $2\pi i$. Let $\mathcal{H}_{k}\subset\mathcal{P}_{{\rm MT}(\mathbb{Z})}^{\mathfrak{m}}$
be the subspace spanned by all weight $k$ motivic multiple zeta values,
and $\mathcal{H}\coloneqq\bigoplus_{k=0}^{\infty}\mathcal{H}_{k}$
the graded ring of motivic multiple zeta values. Put $\mathcal{A}\coloneqq\mathcal{H}/\zeta^{\mathfrak{m}}(2)\mathcal{H}$.
We denote the image of $I^{\mathfrak{m}}(x;w;y)$ in $\mathcal{A}$
by $I^{\mathfrak{a}}(x;w;y)$. By the Hopf structure of $\mathcal{P}_{{\rm MT}(\mathbb{Z})}^{\mathfrak{m}}$,
$\mathcal{A}$ becomes a commutative non-cocommutative Hopf algebra
and $\mathcal{H}$ becomes $\mathcal{A}$-comodule. Its coproduct
(resp. coaction) structure $\Delta:\mathcal{A}\to\mathcal{A}\otimes\mathcal{A}$
(resp. $\mathcal{H}\to\mathcal{A}\otimes\mathcal{H}$) was explicitly
computed by Goncharov (resp. by Brown) and is given by
\[
\Delta I^{\bullet}(a_{0};a_{1},\dots,a_{k};a_{k+1})=\sum_{r=0}^{k}\sum_{0=i_{0}<i_{1}<\cdots<i_{r}<i_{r+1}=k+1}\prod_{j=0}^{r}I^{\mathfrak{a}}(a_{i_{j}};a_{i_{j}+1},\dots,a_{i_{j+1}-1};a_{i_{j+1}})\otimes I^{\bullet}(a_{i_{0}};a_{i_{1}},\dots,a_{i_{r}};a_{i_{r+1}})
\]
for $\bullet=\mathfrak{a}$ (resp. $\bullet=\mathfrak{m}$). Let $\mathcal{U}\coloneqq\left(\mathbb{Q}\left\langle f_{3},f_{5},f_{7},\dots\right\rangle ,\shuffle,{\rm dec}\right)$
be the graded commutative non-cocommutative Hopf algebra whose product
and coproduct are given by shuffle product and deconcatenation respectively,
and the degree of $f_{i}$ is defined to be $i$. We regard $\mathcal{U}\otimes\mathbb{Q}[f_{2}]$
as the graded Hopf comodule of $\mathcal{U}$ by $\Delta(f_{2})=1\otimes f_{2}$.
Then the graded Hopf algebra $\mathcal{A}$ (resp. graded Hopf comodule
$\mathcal{H}$) is non-canonically isomorphic to $\mathcal{U}$ (resp.
$\mathcal{U}\otimes\mathbb{Q}[f_{2}]$). We can show that
\begin{equation}
\{u\in\mathcal{H}_{k}\mid\Delta u=\left(u\mod\zeta^{\mathfrak{m}}(2)\right)\otimes1+1\otimes u,\ {\rm per}(u)=0\}=\{0\},\label{eq:coprod_lemma}
\end{equation}
and this gives a useful necessary and sufficient condition that a
given $\mathbb{Q}$-linear sum of motivic iterated integrals becomes
zero.

\subsection{Coaction and antipode formulas for motivic multiple zeta values}

One of the main tools to study motivic multiple zeta values is their
coproduct structure. Define a linear map $\psi:\mathbb{Q}[x_{0},x_{1},x_{2}]\to\tilde{\mathcal{H}}$
by 
\[
\psi(x_{0}^{k_{0}}x_{1}^{k_{1}}x_{2}^{k_{2}})=k_{0}!k_{1}!k_{2}!J(k_{0};k_{1},k_{2}).
\]
We define an inner product $\left\langle ,\right\rangle _{2}$ on
$\mathbb{Q}[X,Y]$ by 
\[
\left\langle p(X,Y),q(X,Y)\right\rangle _{2}=\left.p(\partial X,\partial Y)q(X,Y)\right|_{X=Y=0},
\]
in other words 
\[
\left\langle X^{a}Y^{b},X^{a'}Y^{b'}\right\rangle _{2}=a!b!\delta_{a,a'}\delta_{b,b'}.
\]
Similarly, we also define an inner product $\left\langle ,\right\rangle _{3}$
on $\mathbb{Q}[x_{0},x_{1},x_{2}]$ by 
\[
\left\langle p(x_{0},x_{1},x_{2}),q(x_{0},x_{1},x_{2})\right\rangle _{3}=\left.p(\partial X,\partial Y,\partial Z)q(X,Y,Z)\right|_{X=Y=Z=0}.
\]
 Furthermore, we denote by $\left\langle ,\right\rangle _{H}:\tilde{\mathcal{H}}\times\tilde{\mathcal{H}}^{\vee}\to\mathbb{Q}$
the natural pairing map.
\begin{lem}
\label{lem:basic_lemma}For $p\in\mathbb{Q}[x_{0},x_{1},x_{2}]$ and
$q\in V_{w}^{+}$, we have
\begin{align*}
\left\langle \psi(p(x_{0},x_{1},x_{2})),\lambda(q(X,Y))\right\rangle _{H} & =\left\langle p(x_{0},x_{1},x_{2}),q(x_{1}-x_{0},x_{2}-x_{0})+q(x_{2}-x_{1},x_{1}-x_{0})-q(x_{2}-x_{1},x_{2}-x_{0})\right\rangle _{3}\\
 & =\left\langle p(-X-Y,X,Y)+p(-Y,-X+Y,X)-p(-Y,-X,X+Y),q(X,Y)\right\rangle _{2}.
\end{align*}
\end{lem}

\begin{proof}
Recall that $J(k_{0};k_{1},\dots,k_{d})\in\tilde{\mathcal{H}}$ is
defined by a motivic iterated integral
\[
J(k_{0};k_{1},\dots,k_{d})=I^{\mathfrak{m}}(0';0^{k_{0}}10^{k_{1}}\cdots10^{k_{d}};1').
\]
Brown's coproduct formula is for iterated integral with tangential
base points $\vec{1}_{0}$ and $-\vec{1}_{1}$, but the same formula
also holds for iterated integral with tangential base points $0'$
and $1'$, i.e.,
\begin{equation}
\Delta I^{\mathfrak{m}}(a_{0}';a_{1},\dots,a_{k};a_{k+1}')=\sum_{r=0}^{k}\sum_{0=i_{0}<i_{1}<\cdots<i_{r}<i_{r+1}=k+1}\prod_{j=0}^{r}I^{\mathfrak{a}}(a_{i_{j}}';a_{i_{j}+1},\dots,a_{i_{j+1}-1};a_{i_{j+1}}')\otimes I^{\mathfrak{m}}(a_{i_{0}}';a_{i_{1}},\dots,a_{i_{r}};a_{i_{r+1}}')\label{eq:coprod_tang_pt}
\end{equation}
for $a_{0},\dots,a_{k+1}\in\{0,1\}$. We put
\[
\mathfrak{J}(t_{0};t_{1},\dots,t_{d}):=\sum_{k_{0},\dots,k_{d}=0}^{\infty}t_{0}^{k_{0}}\cdots t_{d}^{k_{d}}J(k_{0};k_{1},\dots,k_{d}).
\]
Then by (\ref{eq:coprod_tang_pt}) we have
\begin{align}
\Delta\mathfrak{J}(x_{0};x_{1},x_{2}) & =1\otimes\mathfrak{J}(x_{0};x_{1},x_{2})+\bar{\mathfrak{J}}(x_{0};x_{1},x_{2})\otimes1\nonumber \\
 & \ +\bar{\mathfrak{J}}(x_{0};x_{1})\otimes\mathfrak{J}(x_{0};x_{2})+\bar{\mathfrak{J}}(x_{1};x_{2})\otimes\mathfrak{J}(x_{0};x_{1})-\bar{\mathfrak{J}}(-x_{2};-x_{1})\otimes\mathfrak{J}(x_{0};x_{2})\label{eq:coprod_for_gen_ser}
\end{align}
where $\bar{\mathfrak{J}}(-;-)\coloneqq\left(\mathfrak{J}(-;-)\bmod\zeta^{\mathfrak{m}}(2)\right)$.
(See \cite[Theorem 5.1]{Gon_galois_symmetries} for detail, where
Goncharov gives an explicit formula for the coproduct of such type
generating function at the level of formal symbols for general depth,
and the case $m=2$, $(a_{0},a_{1},a_{2},a_{3})=(0,1,1,1)$ of \cite[Theorem 5.1]{Gon_galois_symmetries}
and (\ref{eq:coprod_tang_pt}) give (\ref{eq:coprod_for_gen_ser}).
Since $\left\langle \mathfrak{J}(x;y),\sigma_{n+1}\right\rangle =\left\langle \mathfrak{J}(0;y-x),\sigma_{n+1}\right\rangle =(y-x)^{n}$
and $\left\langle \mathfrak{J}(x_{0};x_{1},x_{2}),\sigma_{m+1}\sigma_{n+1}\right\rangle _{H}=\sum_{j}\left\langle f_{j},\sigma_{m+1}\right\rangle \cdot\left\langle g_{j},\sigma_{n+1}\right\rangle $
with $\mathfrak{J}(x_{0};x_{1},x_{2})=\sum_{j}f_{j}\otimes g_{j}$,
(\ref{eq:coprod_for_gen_ser}) gives
\[
\left\langle \mathfrak{J}(x_{0};x_{1},x_{2}),\lambda(q(X,Y))\right\rangle _{H}=q(x_{1}-x_{0},x_{2}-x_{0})+q(x_{2}-x_{1},x_{1}-x_{0})-q(x_{2}-x_{1},x_{2}-x_{0}).
\]
Thus,
\begin{align*}
\left\langle \psi(p(x_{0},x_{1},x_{2})),\lambda(q(X,Y))\right\rangle _{H} & =\left\langle \left\langle \mathfrak{J}(x_{0};x_{1},x_{2}),p(x_{0},x_{1},x_{2})\right\rangle _{3},\lambda(q(X,Y))\right\rangle _{H}\\
 & =\left\langle p(x_{0},x_{1},x_{2}),\left\langle \mathfrak{J}(x_{0};x_{1},x_{2}),\lambda(q(X,Y))\right\rangle _{H}\right\rangle _{3}\\
 & =\left\langle p(x_{0},x_{1},x_{2}),q(x_{1}-x_{0},x_{2}-x_{0})+q(x_{2}-x_{1},x_{1}-x_{0})-q(x_{2}-x_{1},x_{2}-x_{0})\right\rangle _{3}.
\end{align*}
Hence the first equality of the lemma is proved. The second equality
is obvious by the definition.
\end{proof}
It is known that for $u\in\tilde{\mathcal{H}}$, $\Delta(u)-1\otimes u-u\otimes1=0$
if and only if $u\in\mathfrak{D}_{1}\tilde{\mathcal{H}}.$ As corollaries
of the above lemma, we have the following lemmas.
\begin{lem}
\label{lem:vanish_condition}Let $p=p(x_{0},x_{1},x_{2})\in\mathbb{Q}[x_{0},x_{1},x_{2}]$
be a homogeneous polynomial of degree $w$. Then $\psi(p)\in\mathbb{Q}J(w+1)$
if and only if $p(-X-Y,X,Y)+p(-Y,-X+Y,X)-p(-Y,-X,X+Y)\in V_{w}^{-}$.
\end{lem}

\begin{lem}
\label{lem:dual_vanish_condition}Let $w$ be a positive integer,
$N_{0},N_{1},N_{2}$ subsets of $\mathbb{Z}_{\geq0}$, and $q(X,Y)\in V_{w}^{+}$.
Then $\left.\lambda(q)\right|_{H_{w}(N_{0};N_{1},N_{2})}=0$ if and
only if for all $n_{0}\in N_{0},n_{1}\in N_{1},n_{2}\in N_{2}$, the
coefficients of $x_{0}^{n_{0}}x_{1}^{n_{1}}x_{2}^{n_{2}}$ in 
\[
q(x_{1}-x_{0},x_{2}-x_{0})+q(x_{2}-x_{1},x_{1}-x_{0})-q(x_{2}-x_{1},x_{2}-x_{0})
\]
vanish.
\end{lem}

\begin{lem}
\label{lem:antipode}Let $w$ be a nonnegative even integer, $p=p(x_{0},x_{1},x_{2})\in\mathbb{Q}[x_{0},x_{1},x_{2}]$
be a homogeneous polynomial of degree $w$, and $q(X,Y)\in V_{w}^{+}$.
Then
\[
\left\langle \psi(p(x_{0},x_{1},x_{2})),\lambda(q(X,Y))\right\rangle _{H}=\left\langle \psi(p(x_{1},x_{0},x_{2})),\lambda(q(Y,X))\right\rangle _{H}.
\]
\end{lem}

\begin{lem}
\label{lem:0_even__to_1_even__}Let $a\in2\mathbb{Z}_{\geq0}$ and
$b\in\mathbb{Z}_{\geq0}$. Then for $q\in\mathbb{Q}[X^{2},Y]$, 
\[
\left\langle J(1;a,b),\lambda(q(X,Y))\right\rangle _{H}=\left\langle J(0;a,b),\lambda\left(\frac{\partial}{\partial Y}q(X,Y)\right)\right\rangle _{H}.
\]
\end{lem}

Recall that $\tilde{\mathcal{A}}=\tilde{\mathcal{H}}/\zeta^{\mathfrak{m}}(2)\tilde{\mathcal{H}}$
has a structure of a Hopf algebra. Let $\mathbb{S}:\tilde{\mathcal{A}}\to\tilde{\mathcal{A}}$
be the antipode of this Hopf algebra. Note that $\mathbb{S}$ is an
automorphism of $\tilde{\mathcal{A}}$. By Lemma \ref{lem:antipode},
we have the following lemma.
\begin{lem}
\label{lem:antipode_formula}For nonnegative integers $a$, $b$ and
$c$ such that $a+b+c$ is even, we have
\[
\mathbb{S}(J_{\mathfrak{D}}(a;b,c))=J_{\mathfrak{D}}(b;a,c).
\]
\end{lem}

\subsection{\label{subsec:modularforms}Preliminaries for modular forms}

For a congruence subgroup $\Gamma$ of ${\rm SL}_{2}(\mathbb{Z})$,
we denote by $S_{k}(\Gamma)$ (resp. $M_{k}(\Gamma)$) the $\mathbb{C}$-linear
vector space of weight $k$ cusp (resp. modular) forms for $\Gamma$.
Many modular phenomena for multiple zeta values occur via the period
polynomials of modular forms, which are defined as follows.

Let $w$ be a nonnegative even integer and $k=w+2$. For a level $m$
congruence subgroup $\Gamma$ and $f(z)=\sum_{n\in\frac{1}{m}\mathbb{Z}}a_{n}e^{2\pi inz}\in M_{k}(\Gamma)$,
we put $L(f,s):=\sum_{n\in\frac{1}{m}\mathbb{Z}_{>0}}a_{n}n^{-s}$
and $\tilde{L}(f,s):=(2\pi)^{-s}\Gamma(k-s)^{-1}L(f,s)$. Note that
$\tilde{L}(f,n)=0$ for integers $n$ except for the case $0\leq n\leq k$.
We denote by $V_{w}\subset\mathbb{Q}[X,Y]$ the space of homogeneous
polynomial of degree $w$ in $X$ and $Y$. We define an even (resp.
odd) period polynomial $P_{f}^{+}(X,Y)$ (resp. $P_{f}^{-}(X,Y)$
by the sum
\[
k!\sum_{n}\tilde{L}(f,n)X^{n-1}Y^{k-n-1}
\]
where $n$ runs all odd (resp. even) integers between $0$ and $k$.
Then $P_{f}^{+}(X,Y)\in V_{w}^{+}\otimes\mathbb{C}$ and $P_{f}^{-}(X,Y)\in\frac{1}{XY}V_{w+2}^{+}\otimes\mathbb{C}$.
Furthermore, $P_{f}^{-}(X,Y)\in V_{w}\otimes\mathbb{C}$ if and only
if $f$ vanishes at two cusps $0$ and $\infty$. We define the space
of even (resp. odd) period polynomials with complex or rational coefficients
by
\[
W_{w,\mathbb{C}}^{+,\Gamma}:=\{P_{f}^{+}(X,Y)\mid f\in M_{w+2}(\Gamma)\},\ W_{w,\mathbb{C}}^{-,\Gamma}:=\{P_{f}^{-}(X,Y)\mid f\in\tilde{S}_{w+2}(\Gamma)\},
\]
\[
W_{w}^{+,\Gamma}:=W_{w,\mathbb{C}}^{+,\Gamma}\cap\mathbb{Q}[X,Y],\ W_{w}^{-,\Gamma}=W_{w}^{-,\Gamma}\cap\mathbb{Q}[X,Y]
\]
where $\tilde{S}_{w+2}(\Gamma)$ is the subspace of $M_{w+2}(\Gamma)$
consisting modular forms $f$ which vanish at two cusps $0$ and $\infty$.
We simply write $W_{w,\mathbb{C}}^{\pm,{\rm SL}_{2}(\mathbb{Z})}$
and $W_{w}^{\pm,{\rm SL}_{2}(\mathbb{Z})}$ as $W_{w,\mathbb{C}}^{\pm}$
and $W_{w}^{\pm}$, respectively. We define the right action of $\mathbb{Z}[{\rm GL}(2,\mathbb{Z})]$
to $\mathbb{C}(X,Y)$ by
\[
\left.p(X,Y)\right|_{\gamma}=p(aX+bY,cX+dY)\ \ \ \ (\gamma=\left(\begin{array}{cc}
a & b\\
c & d
\end{array}\right)).
\]
We put
\[
\begin{split}\epsilon=\left(\begin{array}{cc}
-1 & 0\\
0 & 1
\end{array}\right),\ S=\left(\begin{array}{cc}
0 & -1\\
1 & 0
\end{array}\right),\ U=\left(\begin{array}{cc}
1 & -1\\
1 & 0
\end{array}\right).\ T=US^{-1}=\left(\begin{array}{cc}
1 & 1\\
0 & 1
\end{array}\right),\\
T'=U^{2}S^{-1}=\left(\begin{array}{cc}
1 & 0\\
1 & 1
\end{array}\right),\ M=SUSU^{2}S=\left(\begin{array}{cc}
-1 & -1\\
2 & 1
\end{array}\right).
\end{split}
\]
We denote by $\Gamma_{A}$, $\Gamma_{B}$ and $\Gamma_{0}(2)$ the
congruence subgroups of ${\rm SL}_{2}(\mathbb{Z})$ defined by
\[
\Gamma_{A}:=\Gamma(2)\sqcup U\Gamma(2)\sqcup U^{2}\Gamma(2),
\]
\[
\Gamma_{B}:=\Gamma(2)\sqcup S\Gamma(2),
\]
and
\[
\Gamma_{0}(2):=\Gamma(2)\sqcup T\Gamma(2)
\]
where $\Gamma(2)$ is the principal congruence subgroup of level 2.
Then we have
\[
W_{w}^{+}=\{p\in V_{w}^{+}:\left.p\right|_{1+S}=\left.p\right|_{1+U+U^{2}}=0\},
\]
\[
W_{w}^{+,\Gamma_{0}(2)}=\{p\in V_{w}^{+}:\left.p\right|_{(1-T)(1+M)}=0\},
\]
\[
W_{w}^{+,\Gamma_{A}}=\{P\in V_{w}^{+}:\left.p\right|_{1+U+U^{2}}=\left.p\right|_{S+SU+SU^{2}}=0\},
\]
\[
W_{w}^{-}=\{p\in V_{w}^{-}:\left.p\right|_{1+S}=\left.p\right|_{1+U+U^{2}}=0\},
\]
\[
W_{w}^{-,\Gamma_{A}}=\{p\in V_{w}^{-}:\left.p\right|_{1+U+U^{2}}=\left.p\right|_{S+SU+SU^{2}}=0\},
\]
\[
W_{w}^{-,\Gamma_{B}}=\{p\in V_{w}^{-}:\left.p\right|_{1+S}=0\},
\]
and for these cases, $W_{w,\mathbb{C}}^{\pm,\Gamma}=W_{w}^{+,\Gamma}\otimes\mathbb{C}$
and the maps $f\mapsto P_{f}^{\pm}$ are injective (see Lemma \ref{prop:period_polynomials}).

\section{\label{sec:gamma_a}The case $J({\rm odd};{\rm odd},0)$}

The following two theorems (Theorems \ref{thm:GammaA_dual} and \ref{thm:Gamma_A_Rel})
says that (\ref{eq:dual_odd_odd_0}) and (\ref{eq:exact_odd_odd_0})
are exact.
\begin{thm}
\label{thm:GammaA_dual}Let $w$ be a positive even integer. Then
there is an exact sequence
\[
0\to W_{w}^{+,\Gamma_{A}}\xrightarrow{{\rm id}}V_{w}^{+}\xrightarrow{\lambda}H_{w}({\bf odd};{\bf odd},{\bf 0})^{\vee}\to0.
\]
\end{thm}

\begin{proof}
It is enough to check that $\ker\lambda=W_{w}^{+,\Gamma_{A}}$. Let
$p(X,Y)\in V_{w}^{+}$. By Lemma \ref{lem:dual_vanish_condition},
$\lambda(p(X,Y))=0$ if and only if
\[
\left.p(X,Y)\right|_{(T'+T)(1-\epsilon)}=0.
\]
Since $\left.p(X,Y)\right|_{\epsilon}=p(X,Y)$, we have
\begin{align*}
\left.p(X,Y)\right|_{(T'+T)(1-\epsilon)} & =\left.p(X,Y)\right|_{(U+U^{2}-SUS-SU^{2}S)S}.
\end{align*}
Thus the condition $\left.p(X,Y)\right|_{(T'+T)(1+\epsilon)}=0$ is
equivalent to
\[
\left.p(X,Y)\right|_{1+U+U^{2}}=\left.p(X,Y)\right|_{1+(SUS)+SU^{2}S}.
\]
Thus $W_{w}^{+,\Gamma_{A}}\subset\ker\lambda$. On the other hand
if $p(X,Y)\in\ker\lambda$ then $q=\left.p(X,Y)\right|_{1+U+U^{2}}=\left.p(X,Y)\right|_{1+(SUS)+SU^{2}S}$
vanishes since $\left.q\right|_{U}=q$ and $\left.q\right|_{SUS}=q$.
Thus $\ker\lambda\subset W_{w}^{+,\Gamma_{A}}$. Hence the theorem
is proved.
\end{proof}
\begin{thm}
\label{thm:Gamma_A_Rel}Let $w$ be a nonnegative even integer. Then
there is an exact sequence
\[
0\to W_{w}^{-,\Gamma_{A}}\xrightarrow{{\rm id}}V_{w}^{-}\xrightarrow{\phi}H_{w}({\bf odd};{\bf odd},{\bf 0})\to0
\]
where $\phi$ is a linear map defined by $\phi(X^{a}Y^{b})=a!b!J_{\mathfrak{D}}(a;b,0)$.
\end{thm}

\begin{proof}
It is enough to check $\ker\phi={\rm Im}u$. Let $p\in V_{w}^{-}$.
By Lemma \ref{lem:vanish_condition}, $\phi(p)=0$ is and only if
\begin{align*}
0 & =\left.p\right|_{(T+T')(1+\epsilon)}\\
 & =\left.p\right|_{(1+U+U^{2})S}-\left.p\right|_{(1+SUS+SU^{2}S)S}.
\end{align*}
Thus $W_{w}^{-,\Gamma_{A}}\subset\ker\phi$. On the other hand if
$p\in\ker\phi$ then $q=\left.p\right|_{1+U+U^{2}}=\left.p\right|_{1+(SUS)+SU^{2}S}$
vanishes since $\left.q\right|_{U}=q$ and $\left.q\right|_{SUS}=q$.
Thus $\ker\phi\subset W_{w}^{-,\Gamma_{A}}$. Hence the theorem is
proved.
\end{proof}
\begin{example}
By Theorem \ref{thm:Gamma_A_Rel}, for a nonnegative even integer
$w$, an odd period polynomial of a weight $w+2$ modular form for
$\Gamma_{A}$ gives a linear relation among $\{J_{\mathfrak{D}}(a;b,0)\mid a,b\in1+2\mathbb{Z}_{\geq0},a+b=w\}$.
For example $w=4$, the space of weight $w+2$ cusp forms for $\Gamma_{A}$
is generated by the square roots of Ramanujan delta $\sqrt{\Delta}$,
and its odd period polynomial is $c(XY^{3}-XY^{3})$ where $c\in\mathbb{C}$.
Thus $1!3!J(1;3,0)-3!1!J(3;1,0)\in\mathbb{Q}\zeta^{\mathfrak{m}}(6)$.
In fact,
\[
1!3!J(1;3,0)-3!1!J(3;1,0)=11\zeta^{\mathfrak{m}}(6).
\]
See Table \ref{table:gamma_a} for other weights.
\end{example}

\begin{table}
{\footnotesize{}}%
\begin{tabular}{|c|c|c|c|}
\hline 
{\footnotesize{}$w$} & {\footnotesize{}}%
\begin{tabular}{c}
{\footnotesize{}generator of}\tabularnewline
{\footnotesize{}$S_{w+2}(\Gamma_{A})$}\tabularnewline
\end{tabular} & {\footnotesize{}}%
\begin{tabular}{c}
{\footnotesize{}generator of}\tabularnewline
{\footnotesize{}$W_{w}^{-,\Gamma_{A}}$}\tabularnewline
\end{tabular} & {\footnotesize{}relations ($F(a,b):=a!b!J(a;b,0)$)}\tabularnewline
\hline 
\hline 
{\footnotesize{}$4$} & {\footnotesize{}$\sqrt{\Delta}$} & {\footnotesize{}$XY^{3}-X^{3}Y$} & {\footnotesize{}$F(1,3)-F(3,1)=11\zeta^{\mathfrak{m}}(6)$}\tabularnewline
\hline 
{\footnotesize{}$8$} & {\footnotesize{}$E_{4}\sqrt{\Delta}$} & {\footnotesize{}$XY^{7}-2X^{3}Y^{5}+2X^{5}Y^{3}-X^{7}Y$} & {\footnotesize{}}%
\begin{tabular}{l}
{\footnotesize{}$F(1,7)-2F(3,5)+2F(5,3)-F(7,1)$}\tabularnewline
{\footnotesize{}$=2^{3}3^{2}5\cdot29\zeta^{\mathfrak{m}}(10)$}\tabularnewline
\end{tabular}\tabularnewline
\hline 
\multirow{2}{*}{{\footnotesize{}$10$}} & {\footnotesize{}$E_{6}\sqrt{\Delta}$} & {\footnotesize{}$2XY^{9}-3X^{3}Y^{7}+3X^{7}Y^{3}-2X^{9}Y$} & {\footnotesize{}}%
\begin{tabular}{l}
{\footnotesize{}$2F(1,9)-3F(3,7)+3F(7,3)-2F(9,1)$}\tabularnewline
{\footnotesize{}$=\frac{2^{4}3^{3}5\cdot7^{2}10243}{691}\zeta^{\mathfrak{m}}(12)$}\tabularnewline
\end{tabular}\tabularnewline
\cline{2-4} \cline{3-4} \cline{4-4} 
 & {\footnotesize{}$\Delta$} & {\footnotesize{}$4XY^{9}-25X^{3}Y^{7}+42X^{5}Y^{5}-25X^{7}Y^{3}+4X^{9}Y$} & {\footnotesize{}}%
\begin{tabular}{l}
{\footnotesize{}$4F(1,9)-25F(3,7)+42F(5,5)-25F(7,3)+4F(9,1)$}\tabularnewline
{\footnotesize{}$=2^{5}3^{4}5\cdot7\cdot11\zeta^{\mathfrak{m}}(12)$}\tabularnewline
\end{tabular}\tabularnewline
\hline 
\end{tabular}\caption{\label{table:gamma_a}Example of Theorem \ref{thm:Gamma_A_Rel}}
\end{table}
\begin{example}
The even period polynomials for $\sqrt{\Delta}\in M_{4+2}(\Gamma_{A})$
and $E_{4}\sqrt{\Delta}\in M_{8+2}(\Gamma_{A})$ are constant multiples
of $Y^{4}-4X^{2}Y^{2}+X^{4}$ and $3Y^{8}-16X^{2}Y^{6}+20X^{4}Y^{4}-16X^{6}Y^{2}+3X^{8}$,
respectively. Thus by Theorem \ref{thm:GammaA_dual}, we have
\[
\left\langle J(a;b,0),\sigma_{1}\sigma_{5}-4\sigma_{3}\sigma_{3}+\sigma_{5}\sigma_{1}\right\rangle =0\ \ \ (a,b:{\rm odd})
\]
and
\[
\left\langle J(a;b,0),3\sigma_{1}\sigma_{9}-16\sigma_{3}\sigma_{7}+20\sigma_{5}\sigma_{5}-16\sigma_{7}\sigma_{3}+3\sigma_{9}\sigma_{1}\right\rangle =0\ \ \ (a,b:{\rm odd}).
\]
\end{example}

\section{\label{sec:strange_gkz}The case $J({\rm even};{\rm even},0)$}

Let $w$ be a nonnegative even integer. We put
\[
H_{w}^{{\rm sym}}:=\left\langle J_{\mathfrak{D}}(a;w-a,0)+J_{\mathfrak{D}}(w-a;a,0)\mid0\leq a\leq w\right\rangle 
\]
\[
H_{w}^{{\rm asym}}:=\left\langle J_{\mathfrak{D}}(a;w-a,0)-J_{\mathfrak{D}}(w-a;a,0)\mid0\leq a\leq w\right\rangle ,
\]
\[
H_{w,{\rm odd}}^{{\rm sym}}:=\left\langle J_{\mathfrak{D}}(a;w-a,0)+J_{\mathfrak{D}}(w-a;a,0)\mid0<a<w,a:{\rm odd}\right\rangle \subset H_{w}^{{\rm sym}}\cap H_{w}({\bf odd};{\bf odd},{\bf 0}),
\]
\[
H_{w,{\rm even}}^{{\rm sym}}:=\left\langle J_{\mathfrak{D}}(a;w-a,0)+J_{\mathfrak{D}}(w-a;a,0)\mid0\leq a\leq w,a:{\rm even}\right\rangle \subset H_{w}^{{\rm sym}}\cap H_{w}({\bf even};{\bf even},{\bf 0}).
\]

\begin{lem}
\label{lem:ee0_l1}We have
\[
H_{w,{\rm odd}}^{{\rm sym}}\subset H_{w,{\rm even}}^{{\rm sym}}.
\]
\end{lem}

\begin{proof}
We denote by $V_{w}^{-,{\rm sym}}$ the set of symmetric polynomials
in $V_{w}^{-}$. The lemma is equivalent to the congruence
\[
\psi(g(X,Y))\equiv0\pmod{H_{w,{\rm even}}^{{\rm sym}}}
\]
for $g(X,Y)\in V_{w}^{-,{\rm sym}}$. We prove this congruence by
using the double shuffle relation. For a degree $w$ homogeneous polynomial
$p(X,Y,Z)$, define $\psi(p)\in H_{w}$ by linearity and
\[
\psi(X^{a}Y^{b}Z^{c})=a!b!c!J_{\mathfrak{D}}(a;b,c).
\]
Let $a$ and $b$ are nonnegative integer such that $a+b=w$. By the
harmonic relation
\[
J_{\mathfrak{D}}(a)J_{\mathfrak{D}}(b)=J_{\mathfrak{D}}(0;a,b)+J_{\mathfrak{D}}(0;b,a),
\]
and the shuffle relation
\[
a!b!J(a)J(b)=\sum_{a_{1}+a_{2}=a}{a \choose a_{1}}a_{1}!(b+a_{2})!J(0;a_{1},b+a_{2})+\sum_{b_{1}+b_{2}=b}{b \choose b_{1}}b_{1}!(a+b_{2})!J(0;b_{1},a+b_{2}),
\]
we have
\[
\psi(Y^{a}Z^{b}+Z^{a}Y^{b}-(Y+Z)^{a}Z^{b}-(Y+Z)^{b}Z^{a})=0.
\]
Since $\psi(p(X,Y,Z))=\psi(p(X,Y,-X-Y))$ for all $p$, we have
\begin{equation}
\psi((X+Y)^{a}(X^{b}-(-Y)^{b})+(X+Y)^{b}(X^{a}-(-Y)^{a}))=0.\label{eq:in_proof_ee0}
\end{equation}
Note that (\ref{eq:in_proof_ee0}) gives linear relations among the
generators of $H_{w}^{{\rm sym}}$ (resp. $H_{w}^{{\rm asym}}$) for
odd (resp. even) $a$ and $b$. Hence, let us assume that $a$ and
$b$ are odd. Let $q(X,Y)$ be a symmetric odd polynomial $X^{a}Y^{b}+X^{b}Y^{a}$.
Then (\ref{eq:in_proof_ee0}) says that
\begin{equation}
\psi(\left.q(X,Y)\right|_{US+U^{2}S})=0.\label{eq:in_proof_ee0_1}
\end{equation}
Since the odd part of $\left.q(X,Y)\right|_{US+U^{2}S}$ is given
by
\begin{align*}
\frac{1}{2}\left.q(X,Y)\right|_{(US+U^{2}S)(1-\epsilon)} & =\frac{1}{2}\left.q(X,Y)\right|_{US+U^{2}S-\epsilon SU^{2}-\epsilon SU}\\
 & =\frac{1}{2}\left.q(X,Y)\right|_{(U+U^{2})(S-1)}
\end{align*}
we have
\[
\psi(\left.q(X,Y)\right|_{(U+U^{2})(S-1)})\equiv0\pmod{H_{w,{\rm even}}^{{\rm sym}}}.
\]
We denote by $\alpha:V_{w}^{-,{\rm sym}}\to V_{w}^{-,{\rm sym}}$
the map defined by $\alpha(p)=\left.p\right|_{(U+U^{2})(S-1)}$. Thus
if $\alpha$ is surjective then the lemma is proved. For $p(X,Y)\in V_{w}^{-,{\rm sym}}\cap\mathbb{Z}[X,Y]$,
by direct calculation, we have
\[
\alpha(\alpha(p(X,Y)))\equiv p(X,Y)\pmod{3\mathbb{Z}[X,Y]}.
\]
Thus $\alpha$ is injective and thus surjective.
\end{proof}
\begin{lem}
\label{lem:Hsym}We have
\[
H_{w}^{{\rm sym}}=\bigoplus_{m=0}^{\left\lfloor w/4\right\rfloor }J_{\mathfrak{D}}(m)J_{\mathfrak{D}}(w-2m).
\]
\end{lem}

\begin{proof}
Let $\mathbb{S}:H_{w}\to H_{w}$ be the antipode map. Then $H_{w}$is
factored to $\pm$-part by the action of $\mathbb{S}$; $H_{w}=H_{w}^{+}\oplus H_{w}^{-}$.
Then we can easily show that
\[
H_{w}^{+}=\bigoplus_{m=0}^{\left\lfloor w/4\right\rfloor }J_{\mathfrak{D}}(2m)J_{\mathfrak{D}}(w-2m).
\]
By Lemma \ref{lem:antipode_formula}, $\mathbb{S}(v)=v$ for $v\in H_{w}^{{\rm sym}}$
and $\mathbb{S}(v)=-v$ for $v\in H_{w}^{{\rm asym}}$. Thus $H_{w}^{{\rm sym}}\subset H_{w}^{+}$
and $H_{w}^{{\rm asym}}\subset H_{w}^{-}$. Since $H_{w}$ is spanned
by $H_{w}^{{\rm sym}}$ and $H_{w}^{{\rm asym}}$, we have
\[
H_{w}^{{\rm sym}}=H_{w}^{+}.
\]
Thus the lemma is proved.
\end{proof}
\begin{thm}
The sequences (\ref{eq:exact_even_even_0}) and (\ref{eq:dual_even_even_0})
are exact.
\end{thm}

\begin{proof}
Define $\phi:V_{w}^{+}\to V_{w}^{+}$ by
\[
\phi(p):=\left.p\right|_{(2-US-U^{2}S-SU-SU^{2})}.
\]
By Lemmas \ref{lem:vanish_condition} and \ref{lem:dual_vanish_condition},
the sequences (\ref{eq:exact_even_even_0}) and (\ref{eq:dual_even_even_0})
are exact if and only if
\[
W_{w}^{+}=\{p\in V_{w}^{+}\mid\phi(p)=0\}.
\]
Since we can easily check
\[
W_{w}^{+}\subset\{p\in V_{w}^{+}\mid\phi(p)=0\},
\]
it is enough to prove $p\in W_{w}^{+}$ for all $p\in\ker(\phi)$.
Decompose $p$ as $p=p^{+}+p^{-}$ where $p^{\pm}=\frac{1}{2}\left.p\right|_{1\pm S}$.
Since
\[
\left.\phi(q)\right|_{S}=\phi(\left.q\right|_{S})
\]
for any $q\in V_{w}^{+}$, we have $\phi(p^{+})=\phi(p^{-})=0$.

By definition, we have $0=\phi(p^{-})=\left.p^{-}\right|_{(1+U+U^{2})(1-S)}$.
Put $q=\left.p^{-}\right|_{1+U+U^{2}}$. Then $\left.q\right|_{U}=q$
and $\left.q\right|_{S}=q$. Thus $q=0$ or $w=0$. Therefore $p^{-}\in W_{w}^{+}$.

Let us prove $p^{+}=0$ by contradiction. Suppose that $p^{+}\neq0$.
Since $\phi(p^{+})=0$, we have
\[
\sum_{m=0}^{w/2}a_{m}(2m)!(w-2m)!J_{\mathfrak{D}}(2m;w-2m,0)=0
\]
where $p^{+}=\sum_{m=0}^{w/2}a_{m}X^{2m}Y^{w-2m}$. This gives a non-trivial
linear relation among $(J_{\mathfrak{D}}(2m;w-2m,0)+J_{\mathfrak{D}}(w-2m;2m,0))$'s,
and thus
\[
\dim H_{w,{\rm even}}^{{\rm sym}}<\{m\in\mathbb{Z}\mid0\leq2m\leq w-2m\}=1+\left\lfloor \frac{w}{4}\right\rfloor .
\]
Since $H_{w,{\rm even}}^{{\rm sym}}=H_{w}^{{\rm sym}}$ by Lemma \ref{lem:ee0_l1},
we have
\[
\dim H_{w}^{{\rm sym}}<1+\left\lfloor \frac{w}{4}\right\rfloor .
\]
This contradicts Lemma \ref{lem:Hsym}. Hence, $p^{+}=0$.

Therefore $p=p^{+}+p^{-}\in V_{w}^{+}$, and the theorem is proved,
\end{proof}
\begin{example}
Let $w$ be a positive even integer. Then even period polynomial for
the weight $w+2$ Eisenstein series is a constant multiple of $X^{w}-Y^{w}$.
Thus by exactness of (\ref{eq:exact_even_even_0}),
\[
J(w;0,0)-J(0;w,0)\in\mathbb{Q}\zeta^{\mathfrak{m}}(w+2).
\]
In fact,
\begin{align*}
J(w;0,0)-J(0;w,0) & =\sum_{a=0}^{w-1}\zeta^{\mathfrak{m}}(a+1,w-a+1)\\
 & =\zeta^{\mathfrak{m}}(w+2).
\end{align*}
\end{example}

\begin{example}
Let $w=10$. The linear relations among $\{J(0;10,0),J(2;8,0),\dots,J(10;0,0),\zeta^{\mathfrak{m}}(12)\}$
are spanned by $J(10;0,0)-J(0;10,0)=\zeta^{\mathfrak{m}}(12)$ and
\[
14J(2;8,0)-9J(4;6,0)+9J(6;4,0)-14J(8;2,0)=-\frac{6248}{691}\zeta^{\mathfrak{m}}(12).
\]
By multiplying a some rational number and taking a quotient by $\mathbb{Q}\zeta^{\mathfrak{m}}(12)$,
we obtain
\begin{equation}
1\times2!8!J(2;8,0)-3\times4!6!J(4;6,0)+3\times6!4!J(6;4,0)-1\times2!8!J(8;2,0)\equiv0\pmod{\mathbb{Q}\zeta^{\mathfrak{m}}(12)}.\label{eq:example_even_even_0_1331}
\end{equation}
Then the coefficients $1,-3,3,1$ in (\ref{eq:example_even_even_0_1331})
coincide with the coefficients of an even period polynomial
\[
X^{2}Y^{8}-3X^{4}Y^{6}+3X^{6}Y^{4}-X^{8}Y^{2}\in W_{10}^{+}.
\]
\end{example}

\section{\label{sec:0_even_even}The cases similar to the case $J(0;{\rm even},{\rm even})$:
$J(1;{\rm even},{\rm even})$, $J({\rm even};0,{\rm even})$, $J({\rm odd};1,{\rm odd})$}

In this section, we prove the exactness of (\ref{eq:exact_1_even_even}),
(\ref{eq:exact_even_0_even}), (\ref{eq:exact_odd_1_odd}), (\ref{eq:dual_1_even_even}),
(\ref{eq:dual_even_0_even}), (\ref{eq:dual_odd_1_odd}). For the
convenience of the reader, we also give a proof of the exactness of
(\ref{eq:exact_0_even_even}) and (\ref{eq:dual_0_even_even}) because
it can be proved simultaneously without any additional effort.
\begin{lem}
\label{lem:injectivity_u_GKZ}Let $w$ be a nonnegative even integer.
Then $u_{{\rm GKZ}}:W_{w}^{\pm}\to V_{w}^{\pm}$ is injective.
\end{lem}

\begin{proof}
Note that 
\[
u_{{\rm GKZ}}(p)=\frac{1}{2}\left.p\right|_{SU^{2}S+U}.
\]
Assume that $u_{{\rm GKZ}}(p)=0$ for some $p\in W_{w}^{\pm}$. Then
\[
0=\left.p\right|_{(SU^{2}S+U)U^{2}}=\left.p\right|_{SU^{2}SU^{2}+1}.
\]
Thus 
\[
\left.p\right|_{(SU^{2})^{4}}=p.
\]
Since $(SU^{2})^{4}=\left(\begin{array}{cc}
1 & -4\\
0 & 1
\end{array}\right),$ $p\in\mathbb{Q}Y^{w}$. Furthermore, since $\left.p\right|_{1+S}=0$,
$p=0$. Thus the lemma is proved.
\end{proof}
\begin{lem}
\label{lem:gkz_type_l1}Let $w$ be a nonnegative even integer. Let
$p(X,Y)\in V_{w}^{\pm}$. Then
\[
p(X,Y)-p(X+Y,Y)\pm p(X+Y,X)\in V_{w}^{\mp},
\]
if and only if $p(X,Y)\in W_{w}^{\pm}$.
\end{lem}

\begin{proof}
We can assume $w\geq2$ since the case $w=0$ is obvious. The `if'
part is obvious, so it is enough to consider the `only if' part. Assume
that 
\begin{equation}
p(X,Y)-p(X+Y,Y)\pm p(X+Y,X)\in V_{w}^{\mp}.\label{eq:e1_lemgkz}
\end{equation}
By swapping $X$ and $Y$ of (\ref{eq:e1_lemgkz}), we have
\begin{equation}
p(Y,X)-p(X+Y,X)\pm p(X+Y,Y)\in V_{w}^{\mp}.\label{eq:e2_lemgkz}
\end{equation}
By (\ref{eq:e1_lemgkz}) and (\ref{eq:e2_lemgkz}), we have $p(X,Y)\pm p(Y,X)\in V_{w}^{\pm}$.
Furthermore, since $p(X,Y)\in V_{w}^{\pm}$, $\left.p(X,Y)\right|_{1+S}=0$.
Thus
\begin{align*}
0 & =\left.p(X,Y)\right|_{(1-T-T')(1\pm\epsilon)}\\
 & =\left.p(X,Y)\right|_{1-US-U^{2}S\pm\epsilon(1-SU^{2}-SU)}\\
 & =\left.p(X,Y)\right|_{1-US-U^{2}S+1-SU^{2}-SU)}\\
 & =\left.p(X,Y)\right|_{(1+U+U^{2})(1-US)}=\left.p(X,Y)\right|_{(1+U+U^{2})(1-U^{2}S)}.
\end{align*}
Thus $\left.p(X,Y)\right|_{1+U+U^{2}}=0$ since there is no nonzero
polynomial annihilated by both $1-US$ and $1-U^{2}S$.
\end{proof}
\begin{thm}
The sequences (\ref{eq:exact_0_even_even}), (\ref{eq:exact_1_even_even}),
(\ref{eq:exact_even_0_even}), (\ref{eq:exact_odd_1_odd}), (\ref{eq:dual_0_even_even}),
(\ref{eq:dual_1_even_even}), (\ref{eq:dual_even_0_even}) and (\ref{eq:dual_odd_1_odd})
are exact.
\end{thm}

\begin{rem}
The exactness of (\ref{eq:exact_0_even_even}) and (\ref{eq:dual_0_even_even})
is not new (and others are new results).
\end{rem}

\begin{proof}
The exactness of (\ref{eq:dual_0_even_even}) follows from Lemmas
\ref{lem:dual_vanish_condition} and \ref{lem:gkz_type_l1}. The exactness
of (\ref{eq:dual_1_even_even}) and (\ref{eq:dual_even_0_even}) follows
from that of (\ref{eq:dual_0_even_even}) and Lemmas \ref{lem:0_even__to_1_even__}
and \ref{lem:antipode}.

Let us check the exactness of (\ref{eq:dual_odd_1_odd}):
\[
0\to(\partial_{X}^{-1}W_{w}^{-})\xrightarrow{{\rm id}}V_{w+1}^{+}\xrightarrow{\lambda}H_{w}({\bf odd};{\bf 1};{\bf odd})^{\vee}\to0,
\]
where $w$ is a nonnegative even integer. For $p(X,Y)\in V_{w+1}^{+}$,
by Lemma \ref{lem:dual_vanish_condition}, $\lambda(p)=0$ if and
only if
\[
p_{1}(Y-X,-X)-p_{1}(-X,Y)+p_{2}(Y,-X)+p_{1}(Y-X,Y)\in V_{w}^{+}
\]
where $p_{1}(Y,X)=\frac{\partial}{\partial X}p(X,Y)$ and $p_{2}(X,Y)=\frac{\partial}{Y}p(X,Y)$.
Thus, $\lambda(p)=0$ if and only if
\[
p_{1}(X,Y)-p_{1}(X+Y,Y)-p_{1}(X+Y,X)\in V_{w}^{+}.
\]
Note that $p_{1}(X,Y)\in V_{w}^{-}$. Therefore, by Lemma \ref{lem:gkz_type_l1},
$\lambda(p)=0$ if and only if $p_{1}(X,Y)\in W_{w}^{-}$. Thus $\lambda(p)=0$
if and only if $\frac{\partial}{\partial X}p(X,Y)\in W_{w}^{-}$.
Thus (\ref{eq:dual_odd_1_odd}) is exact.

Let us check the exactness of (\ref{eq:exact_0_even_even}) and (\ref{eq:exact_odd_1_odd}):
\[
\begin{split}0\to W_{w}^{+}\xrightarrow{u_{{\rm GKZ}}}V_{w}^{+}\xrightarrow{X^{a}Y^{b}\mapsto a!b!J_{\mathfrak{D}}(0;a,b)}H_{w}({\bf 0};{\bf even},{\bf even})\to0,\\
0\to W_{w}^{-}\xrightarrow{u_{{\rm GKZ}}}V_{w}^{-}\xrightarrow{X^{a}Y^{b}\mapsto a!b!J_{\mathfrak{D}}(a;1,b)}H_{w+1}({\bf odd};{\bf 1},{\bf odd})\to0.
\end{split}
\]
By Lemma \ref{lem:vanish_condition} and direct calculation, we can
show that these sequences are chain complex. By definition and Lemma
\ref{lem:injectivity_u_GKZ}, the left and right parts of these sequences
are exact, i.e., $0\to W_{w}^{+}\xrightarrow{u_{{\rm GKZ}}}V_{w}^{+}$,
$0\to W_{w}^{-}\xrightarrow{u_{{\rm GKZ}}}V_{w}^{-}$, $V_{w}^{+}\to H_{w}({\bf 0};{\bf even},{\bf even})\to0$,
and $V_{w}^{-}\to H_{w+1}({\bf odd};{\bf 1},{\bf odd})\to0$ are exact.
Furthermore, by exactness of (\ref{eq:dual_0_even_even}) and (\ref{eq:dual_odd_1_odd}),
\[
\dim H_{w}({\bf 0};{\bf even};{\bf even})=\frac{w}{2}+1-\dim W_{w}^{+},\ \ \dim H_{w+1}({\bf odd};{\bf 1};{\bf odd})=\frac{w}{2}-\dim W_{w}^{-}.
\]
Thus (\ref{eq:exact_0_even_even}) and (\ref{eq:exact_odd_1_odd})
are exact.

By Lemmas \ref{lem:0_even__to_1_even__} and \ref{lem:antipode},
the exactness of (\ref{eq:exact_1_even_even}) and (\ref{eq:exact_even_0_even})
follows from that of (\ref{eq:exact_0_even_even}).
\end{proof}
\begin{example}
In \cite[Example of Theorem 3]{GKZ}, the authors give two examples
of linear relations coming from the GKZ-relation:
\[
28\zeta(3,9)+150\zeta(5,7)+168\zeta(7,5)=\frac{5197}{691}\zeta(12),
\]
\[
66\zeta(3,13)+375\zeta(5,11)+686\zeta(7,9)+675\zeta(9,7)+396\zeta(11,5)=\frac{78967}{3617}\zeta(16).
\]
The exactness of (\ref{eq:exact_even_even_0}) and (\ref{eq:exact_1_even_even})
implies that $J({\rm even};{\rm even},0)$'s and $J({\rm 1};{\rm even},{\rm even})$'s
also satisfy very similar linear relations where the coefficients
perfectly coincide (except for those of single zetas). In fact,
\begin{align*}
28J(2;0,8)+150J(4;0,6)+168J(6;0,4) & =\frac{118492}{691}\zeta^{\mathfrak{m}}(12),\\
28J(1;2,8)+150J(1;4,6)+168J(1;6,4) & =-69\zeta^{\mathfrak{m}}(13),\\
66J(2;0,12)+375J(4;0,10)+686J(6;0,8)+675J(8;0,6)+396J(10;0,4) & =\frac{3961222}{3617}\zeta^{\mathfrak{m}}(16),\\
66J(1;2,12)+375J(1;4,10)+686J(1;6,8)+675J(1;8,6)+396J(1;10,4) & =-283\zeta^{\mathfrak{m}}(17).
\end{align*}
\end{example}

\begin{example}
Let us see the $w=10$ case of (\ref{eq:exact_odd_1_odd}). The linear
relations among $\{J(1;1,9),J(3;1,7),\dots,J(9;1,1),\zeta^{\mathfrak{m}}(13)\}$
are given by 
\[
48J(1;1,9)+119J(3;1,7)+10J(5;1,5)-144J(7;1,3)=640\zeta^{\mathfrak{m}}(13).
\]
By multiplying $-30240$ and taking a quotient by $\mathbb{Q}\zeta^{\mathfrak{m}}(13)$,
we obtain
\begin{equation}
-4\times1!9!J(1;1,9)-119\times3!7!J(3;1,7)-21\times5!5!J(5;1,5)+144\times7!3!J(7;1,3)\equiv0\pmod{\mathbb{Q}\zeta^{\mathfrak{m}}(13)}.\label{eq:ex_o_1_o}
\end{equation}
On the other hand, the space $W_{10}^{-}$ is spanned by 
\[
p(X,Y):=4XY^{9}-25X^{3}Y^{7}+42X^{5}Y^{5}-25X^{7}Y^{3}+4X^{9}Y.
\]
Then 
\begin{equation}
u_{{\rm GKZ}}(p(X,Y))=\frac{p(-X-Y,X)+p(X-Y,X)}{2}=-4XY^{9}-119X^{3}Y^{7}-21X^{5}Y^{5}+144X^{7}Y^{3}.\label{eq:ex_o_1_o_}
\end{equation}
Thus the same coefficients $-4,-119,-21,144$ appears in both (\ref{eq:ex_o_1_o})
and (\ref{eq:ex_o_1_o_}).
\end{example}

\section{\label{sec:0_even_odd_DingMa}The cases similar to the case $J(0;{\rm even},{\rm odd})$:
$J(1;{\rm even},{\rm odd})$, $J({\rm even};1,{\rm odd})$, $J({\rm odd};0,{\rm even})$}

In this section, we prove the exactness of (\ref{eq:exact_1_even_odd}),
(\ref{eq:exact_even_1_odd}), (\ref{eq:exact_odd_0_even}), (\ref{eq:dual_1_even_odd}),
(\ref{eq:dual_even_1_odd}), (\ref{eq:dual_odd_0_even}). For the
convenience of the reader, we also give a proof of the exactness of
(\ref{eq:exact_0_even_odd}) and (\ref{eq:dual_0_even_odd}) because
it can be proved simultaneously without any additional effort.
\begin{lem}
\label{lem:zagier_lemma}Let $w\in1+2\mathbb{Z}_{\geq0}$, $p(X,Y)\in W_{w-1}^{\mp}$
and $q(X,Y)\in W_{w+1}^{\pm}$. If $Yp(X,Y)+\frac{\partial}{\partial X}q(X,Y)=0$
then $p(X,Y)=q(X,Y)=0$.
\end{lem}

\begin{proof}
By assumption, we have
\begin{equation}
Yp(X,Y)+\frac{\partial}{\partial X}q(X,Y)=0.\label{eq:ding_ma_1}
\end{equation}
By swapping $X$ and $Y$, we have
\begin{align}
0 & =Xp(Y,X)+\frac{\partial}{\partial Y}q(Y,X)\nonumber \\
 & =\pm Xp(X,Y)\mp\frac{\partial}{\partial Y}q(X,Y).\label{eq:ding_ma_2}
\end{align}
By (\ref{eq:ding_ma_1}) and (\ref{eq:ding_ma_2}), we have
\begin{align*}
0 & =X\frac{\partial}{\partial X}q(X,Y)+Y\frac{\partial}{\partial Y}q(X,Y)\\
 & =(w+1)q(X,Y).
\end{align*}
Thus $q(X,Y)=0$, and $p(X,Y)=-\frac{1}{Y}\frac{\partial}{\partial X}q(X,Y)=0$.
Hence the lemma is proved.
\end{proof}
\begin{lem}
\label{lem:ding_ma}Let $w\in1+2\mathbb{Z}_{\geq0}$ and $f(X,Y)\in W_{w}^{\pm}$.
Put $R(X,Y):=f(X+Y,X)$. If
\[
R(X,Y)\pm R(Y,X)\mp R(-X,Y)-R(Y,-X)=0
\]
then
\[
f(X,Y)\in\mathbb{Q}X^{w}.
\]
\end{lem}

\begin{proof}
By assumption, we have
\begin{equation}
0=f(X+Y,X)\pm f(X+Y,Y)\mp f(-X+Y,-X)-f(-X+Y,Y).\label{eq:ding_ma_3}
\end{equation}
By swapping $X$ and $Y$, we have
\begin{align}
0 & =f(X+Y,Y)\pm f(X+Y,X)\mp f(X-Y,-Y)-f(X-Y,X)\nonumber \\
 & =f(X+Y,Y)\pm f(X+Y,X)\pm f(-X+Y,Y)+f(-X+Y,-X)\nonumber \\
\pm0 & =f(X+Y,X)\pm f(X+Y,Y)+f(-X+Y,Y)\pm f(-X+Y,-X).\label{eq:ding_ma_4}
\end{align}
By adding (\ref{eq:ding_ma_3}) and (\ref{eq:ding_ma_4}), we have
\begin{align*}
0 & =f(X+Y,X)\pm f(X+Y,Y)\\
0 & =f(X,Y)-f(X,-X+Y)\ \ \ (X\mapsto Y,Y\mapsto X-Y).
\end{align*}
Thus
\[
f(X,Y)\in\mathbb{Q}X^{w}.
\]
\end{proof}
\begin{lem}
\label{lem:ding_ma_injectivity}Let $w\in1+2\mathbb{Z}_{\geq0}$,
$p(X,Y)\in W_{w-1}^{\pm}$ and $q(X,Y)\in W_{w+1}^{\mp}$. Put $R(X,Y):=Xp(X+Y,X)+\frac{\partial}{\partial Y}q(X+Y,X)$.
If
\[
R(X,Y)\pm R(Y,X)\mp R(-X,Y)-R(Y,X)=0
\]
then $p(X,Y)=0$ and $q(X,Y)\in\mathbb{Q}(X^{w+1})$.
\end{lem}

\begin{proof}
Put 
\[
f(X,Y):=Yp(X,Y)+\frac{\partial}{\partial X}q(X,Y)\in V_{w}^{\pm}.
\]
Then by definition,
\[
R(X,Y)=f(X+Y,X).
\]
Thus by Lemma \ref{lem:ding_ma},
\[
f(X,Y)=\begin{cases}
0 & V_{w}^{\pm}=V_{w}^{+}\\
cX^{w} & V_{w}^{\pm}=V_{w}^{-}
\end{cases}
\]
where $c\in\mathbb{Q}$. Thus
\[
Yp(X,Y)+\frac{\partial}{\partial X}\bar{q}(X,Y)=0
\]
where
\[
\bar{q}(X,Y)=\begin{cases}
q(X,Y) & V_{w}^{\pm}=V_{w}^{+}\\
q(X,Y)-\frac{c}{w+1}(X^{w+1}-Y^{w+1}) & V_{w}^{\pm}=V_{w}^{-}.
\end{cases}
\]
Thus by Lemma \ref{lem:zagier_lemma}, we have $p(X,Y)=0$ and $\bar{q}(X,Y)=0$.
Hence the lemma is proved.
\end{proof}
\begin{lem}
\label{lem:Li_Liu_keyLemma}Let $w\in1+2\mathbb{Z}_{\geq0}$ and $C(X,Y)\in V_{w}^{\pm}$.
Put $D(X,Y)=\frac{\partial}{\partial X}C(X,Y)\pm\frac{\partial}{\partial Y}C(Y,X)$,
$L(X,Y)=\pm C(X+Y,X)+C(X+Y,Y)-C(X,Y)$, $S(X,Y)=\frac{1}{2}\left(L(X,Y)\pm L(-X,Y)\right)$,
$S_{1}(X,Y)=\frac{\partial}{\partial X}\left(S(X,Y)\pm S(Y,X)\right)$,
$S_{2}(X,Y)=\frac{\partial}{\partial X}\left(S(X,Y)\mp S(Y,X)\right)$.
Then
\[
\pm D(X+Y,X)+D(X+Y,Y)-D(X,Y)=S_{1}(X,Y)-S_{2}(X,X+Y).
\]
\end{lem}

\begin{proof}
By linearity, it is enough to only consider the case $C(X,Y)=X^{a}Y^{b}$
where $(-1)^{a}=\pm1$ and $(-1)^{b}=\mp1$. Then
\[
p(X,Y)=a(X^{a-1}Y^{b}\pm Y^{a-1}X^{b}),
\]
and thus
\begin{align*}
 & \pm D(X+Y,X)+D(X+Y,Y)-D(X,Y)\\
 & =a\left(\pm(X+Y)^{a-1}X^{b}+X^{a-1}(X+Y)^{b}+(X+Y)^{a-1}Y^{b}\pm Y^{a-1}(X+Y)^{b}-X^{a-1}Y^{b}\mp Y^{a-1}X^{b}\right).
\end{align*}
On the other hand,
\[
L(X,Y)=\pm(X+Y)^{a}X^{b}+(X+Y)^{a}Y^{b}-X^{a}Y^{b},
\]
\begin{align*}
S(X,Y) & =\frac{1}{2}\left(\pm(X+Y)^{a}X^{b}+(X+Y)^{a}Y^{b}-2X^{a}Y^{b}-(X-Y)^{a}X^{b}+(X-Y)^{a}Y^{b}\right),
\end{align*}
\begin{align*}
S(X,Y)\pm S(Y,X) & =\pm(X+Y)^{a}X^{b}+(X+Y)^{a}Y^{b}-X^{a}Y^{b}\mp Y^{a}X^{b},
\end{align*}
\begin{align*}
S(X,Y)\mp S(Y,X) & =-X^{a}Y^{b}\pm Y^{a}X^{b}-(X-Y)^{a}X^{b}+(X-Y)^{a}Y^{b}.
\end{align*}
\[
S_{1}(X,Y)=\frac{\partial}{\partial X}\left(S(X,Y)\pm S(Y,X)\right)=\pm(X+Y)^{a-1}X^{b-1}(aX+bX+bY)+a(X+Y)^{a-1}Y^{b}-aX^{a-1}Y^{b}\mp bY^{a}X^{b-1}.
\]
\[
S_{2}(X,Y)=-(X-Y)^{a-1}X^{b-1}(aX+bX-bY)+a(X-Y)^{a-1}Y^{b}-aX^{a-1}Y^{b}\pm bY^{a}X^{b-1}.
\]
\[
S_{2}(X,X+Y)=-(-Y)^{a-1}X^{b-1}(aX-bY)+a(-Y)^{a-1}(X+Y)^{b}-aX^{a-1}(X+Y)^{b}\pm b(X+Y)^{a}X^{b-1},
\]
\begin{align*}
S_{1}(X,Y)-S_{2}(X,X+Y) & =a\left(\pm(X+Y)^{a-1}X^{b}+X^{a-1}(X+Y)^{b}+(X+Y)^{a-1}Y^{b}\pm Y^{a-1}(X+Y)^{b}-X^{a-1}Y^{b}\mp Y^{a-1}X^{b}\right).
\end{align*}
Thus the lemma is proved.
\end{proof}
\begin{lem}
\label{lem:Li-Liu_1}Let $w\in1+2\mathbb{Z}_{\geq0}$ and $q(X,Y)\in V_{w+1}^{\pm}$.
Assume that $q(X,Y)=\mp q(Y,X)$ and
\[
\left(\frac{\partial}{\partial Y}-\frac{\partial}{\partial X}\right)\left(q(X,Y)-q(X+Y,Y)\pm q(X+Y,X)\right)\in V_{w}^{\mp}.
\]
Then $q(X,Y)\in\hat{W}_{w+1}^{\pm}$ where
\[
\hat{W}_{w+1}^{+}=W_{w+1}^{+},\ \hat{W}_{w+1}^{-}=W_{w+1}^{-}\oplus\mathbb{Q}(X^{w}+Y^{w}).
\]
\end{lem}

\begin{proof}
By changing $X$ and $Y$,
\[
\left(\frac{\partial}{\partial Y}-\frac{\partial}{\partial X}\right)\left(q(X,Y)-q(X+Y,Y)\pm q(X+Y,X)\right)\in V_{w}^{\pm}.
\]
Thus
\[
\left(\frac{\partial}{\partial Y}-\frac{\partial}{\partial X}\right)\left(q(X,Y)-q(X+Y,Y)\pm q(X+Y,X)\right)=0.
\]
Thus there exists $c\in\mathbb{Q}$ such that
\begin{equation}
q(X,Y)-q(X+Y,Y)\pm q(X+Y,X)=c(X+Y)^{w+1}.\label{eq:e1}
\end{equation}
By changing $X$ and $Y$, we have
\begin{equation}
-q(X,Y)\mp q(X+Y,X)+q(X+Y,Y)=\pm c(X+Y)^{w+1}.\label{eq:e2}
\end{equation}
By considering (\ref{eq:e1}) plus (\ref{eq:e2}), we have
\[
c(1\pm1)=0.
\]
Thus the lemma is proved.
\end{proof}
\begin{lem}
\label{lem:Li-Liu-2}Let $w\in1+2\mathbb{Z}_{\geq0}$ and $C(X,Y)\in V_{w}^{\pm}$.
Then
\[
C(X,Y)-C(X+Y,Y)\mp C(X+Y,X)\in V_{w}^{\mp}
\]
if and only if there exists $p(X,Y)\in W_{w-1}^{\mp}$ and $q(X,Y)\in W_{w+1}^{\pm}$
such that
\[
C(X,Y)=Xp(X,Y)+\frac{\partial}{\partial Y}q(X,Y).
\]
\end{lem}

\begin{proof}
It is enough to prove the `only if' part. Put
\[
p(X,Y):=\frac{1}{w+1}\left(\frac{\partial}{\partial X}C(X,Y)\pm\frac{\partial}{\partial Y}C(Y,X)\right)\in V_{w}^{\mp}.
\]
\[
q(X,Y):=\frac{1}{w+1}\left(YC(X,Y)\mp XC(Y,X)\right)\in V_{w}^{\pm}.
\]
Then by definition,
\begin{align}
Xp(X,Y)+\frac{\partial}{\partial Y}q(X,Y) & =C(X,Y).\label{eq:C_express}
\end{align}
By Lemma \ref{lem:Li_Liu_keyLemma}
\[
p(X,Y)-p(X+Y,Y)\mp p(X+Y,X)=0.
\]
Thus
\[
p(X,Y)\in W_{w}^{\mp}.
\]
By (\ref{eq:C_express}), by modulo $V_{w}^{\mp}$, we have
\begin{align*}
0 & \equiv C(X,Y)-C(X+Y,Y)\mp C(X+Y,X)\\
 & =Xp(X,Y)-(X+Y)p(X+Y,Y)\mp(X+Y)p(X+Y,X)\\
 & \ +\frac{\partial}{\partial Y}q(X,Y)-\left(\frac{\partial}{\partial Y}-\frac{\partial}{\partial X}\right)q(X+Y,Y)\mp\left(\frac{\partial}{\partial X}-\frac{\partial}{\partial Y}\right)q(X+Y,X)\\
 & \equiv(X+Y)\left(p(X,Y)-p(X+Y,Y)\mp p(X+Y,X)\right)\\
 & \ +\left(\frac{\partial}{\partial Y}-\frac{\partial}{\partial X}\right)\left(q(X,Y)-q(X+Y,Y)\pm q(X+Y,X)\right)\\
 & \equiv\left(\frac{\partial}{\partial Y}-\frac{\partial}{\partial X}\right)\left(q(X,Y)-q(X+Y,Y)\pm q(X+Y,X)\right)\pmod{V_{w}^{\mp}}.
\end{align*}
By Lemma \ref{lem:Li-Liu_1}, $q(X,Y)\in\hat{W}_{w+1}^{\pm}$. Furthermore,
we can easily show that $q(X,Y)\in W_{w+1}^{\mp}$. Thus the lemma
is proved.
\end{proof}
\begin{thm}
The sequences (\ref{eq:exact_0_even_odd}), (\ref{eq:exact_1_even_odd}),
(\ref{eq:exact_even_1_odd}), (\ref{eq:exact_odd_0_even}), (\ref{eq:dual_0_even_odd}),
(\ref{eq:dual_1_even_odd}), (\ref{eq:dual_even_1_odd}) and (\ref{eq:dual_odd_0_even})
are exact.
\end{thm}

\begin{rem}
The exactness of (\ref{eq:exact_0_even_odd}) and (\ref{eq:dual_0_even_odd})
is not new (and others are new results).
\end{rem}

\begin{proof}
Let us check the exactness of (\ref{eq:dual_0_even_odd}) and (\ref{eq:dual_odd_0_even}):
\[
\begin{split}0\to W_{w-1}^{-}\oplus W_{w+1}^{+}\xrightarrow{(p,q)\mapsto Xp+\frac{\partial}{\partial Y}q}V_{w}^{+}\xrightarrow{\lambda_{1}}H_{w}({\bf 0};{\bf even};{\bf odd})^{\vee}\to0\\
0\to W_{w-1}^{+}\oplus W_{w+1}^{-}\xrightarrow{(p,q)\mapsto Yp+\frac{\partial}{\partial X}q}V_{w}^{+}\xrightarrow{\lambda_{2}}H_{w}({\bf odd};{\bf 0};{\bf even})^{\vee}\to0
\end{split}
\]
where $\lambda_{1}$ and $\lambda_{2}$ are linear maps induced from
$\lambda:V_{w}^{+}\to H_{w}^{\vee}$. By Lemma \ref{lem:zagier_lemma},
the left parts $0\to W_{w-1}^{-}\oplus W_{w+1}^{+}\to V_{w}^{+}$
and $0\to W_{w-1}^{+}\oplus W_{w+1}^{-}\to V_{w}^{+}$ are exact.
The exactness of the right parts is obvious. By Lemma \ref{lem:dual_vanish_condition},
$C(X,Y)\in\ker(\lambda_{1})$ if and only if
\[
C(X,Y)-C(X+Y,Y)-C(X+Y,X)\in V_{w}^{-}.
\]
Thus (\ref{eq:dual_0_even_odd}) is exact by Lemma \ref{lem:Li-Liu-2}.
On the other hand, by Lemma \ref{lem:dual_vanish_condition}, $C(Y,X)\in\ker(\lambda_{2})$
if and only if
\[
C(X,Y)-C(X+Y,Y)+C(X+Y,X)\in V_{w}^{+}.
\]
Thus (\ref{eq:dual_odd_0_even}) is exact by Lemma \ref{lem:Li-Liu-2}.

The exactness of (\ref{eq:dual_1_even_odd}) and (\ref{eq:dual_even_1_odd})
follows from the exactness of (\ref{eq:dual_0_even_odd}) and Lemmas
\ref{lem:0_even__to_1_even__} and \ref{lem:antipode}.

Let us check the exactness of (\ref{eq:exact_0_even_odd}) and (\ref{eq:exact_odd_0_even}):
\[
\begin{split}0\to W_{w-1}^{-}\times\tilde{W}_{w+1}^{+}\times\mathbb{Q}\xrightarrow{u_{{\rm M}}}V_{w}^{+}\xrightarrow{X^{a}Y^{b}\mapsto a!b!J_{\mathfrak{D}}(0;a,b)}H_{w}({\bf 0};{\bf even},{\bf odd})\to0\\
0\to W_{w-1}^{+}\times W_{w+1}^{-}\xrightarrow{u_{{\rm M}}'}V_{w}^{-}\xrightarrow{X^{a}Y^{b}\mapsto a!b!J_{\mathfrak{D}}(a;0,b)}H_{w}({\bf odd};{\bf 0},{\bf even})\to0.
\end{split}
\]
By Lemma \ref{lem:vanish_condition} and direct calculation, we can
show that these sequences are chain complex. By Lemma \ref{lem:ding_ma_injectivity},
the left parts of these sequences are exact. The exactness of the
right parts is obvious. Furthermore, by exactness of (\ref{eq:dual_0_even_odd})
and (\ref{eq:dual_odd_0_even}), 
\[
\dim H_{w}({\bf 0};{\bf even};{\bf odd})=\frac{w+1}{2}-\dim W_{w-1}^{-}-\dim W_{w+1}^{+},\ \ \dim H_{w}({\bf odd};{\bf 0};{\bf even})=\frac{w+1}{2}-\dim W_{w-1}^{+}-\dim W_{w+1}^{-}.
\]
Thus (\ref{eq:exact_0_even_even}) and (\ref{eq:exact_odd_1_odd})
are exact.

By Lemmas \ref{lem:0_even__to_1_even__} and \ref{lem:antipode},
the exactness of (\ref{eq:exact_1_even_even}) and (\ref{eq:exact_1_even_even})
follows from that of (\ref{eq:exact_0_even_even}).
\end{proof}

\section{\label{sec:0_odd_odd}The cases $J(0,{\rm odd},{\rm odd})$, $J({\rm odd};0,{\rm odd})$}
\begin{thm}[Tasaka]
\label{thm:Tasaka_exact}The sequences (\ref{eq:exact_0_odd_odd})
and (\ref{eq:dual_0_odd_odd}) are exact.
\end{thm}

\begin{proof}
Let $w$ be a nonnegative even integer and $\lambda':V_{w}^{+}\to H_{w}({\bf 0};{\bf odd};{\bf odd})^{\vee}$
the linear map induced from $\lambda:V_{w}^{+}\to H_{w}^{\vee}$.
Then, by Lemma \ref{lem:dual_vanish_condition}, we have
\[
\ker(p)=0\Leftrightarrow p\mid_{U-SU^{2}-SU^{2}S+US}=0.
\]
Since
\[
p\mid_{U-SU^{2}-SU^{2}S+US}=p\mid_{-S(1-T)(1+M)SU},
\]
the sequence (\ref{eq:dual_0_odd_odd}) is exact. The exactness of
(\ref{eq:exact_0_odd_odd}) follows from that of (\ref{eq:dual_0_odd_odd})
and the calculation of the dimension.
\end{proof}
\begin{thm}
The sequences (\ref{eq:exact_odd_0_odd}) and (\ref{eq:dual_odd_0_odd})
are exact.
\end{thm}

\begin{proof}
It follows from Lemma \ref{lem:antipode} and Theorem \ref{thm:Tasaka_exact}.
\end{proof}

\section{\label{sec:no-relation}Almost no relation cases: $J({\rm even};1,{\rm even})$,
$J({\rm even};{\rm odd},0)$, $J({\rm odd};{\rm even},0)$}

We put $J_{!}(m_{0};m_{1},m_{2}):=m_{0}!m_{1}!m_{2}!J_{\mathfrak{D}}(m_{0};m_{1},m_{2})$
and $J_{!}(n):=n!J_{\mathfrak{D}}(0;n)$. By Lemma \ref{lem:vanish_condition},
for an odd integer $w$ and $m_{0}+m_{1}+m_{2}=w$, we have 
\begin{equation}
J_{!}(m_{0},m_{1},m_{2})=\sum_{n=0}^{(w-1)/2}a_{n}J_{!}(2n)J_{!}(w-2n)\label{eq:odd_case}
\end{equation}
where $a_{n}$ is the coefficient of $X^{2n}Y^{w-2n}$ in
\[
(-X-Y)^{m_{0}}X^{m_{1}}Y^{m_{2}}+(-Y)^{m_{0}}(-X+Y)^{m_{1}}X^{m_{2}}-(-Y)^{m_{0}}(-X)^{m_{1}}(X+Y)^{m_{2}}.
\]
Note that $\{J_{!}(2n)J_{!}(w-2n)\}$ is linearly independent.
\begin{thm}
There are no linear relations among $J_{\mathfrak{D}}({\rm even};1,{\rm even})$'s.
\end{thm}

\begin{proof}
Let $w$ be a positive odd integer. By (\ref{eq:odd_case}), 
\[
J_{!}(w-1-2m;1,2m)=\sum_{n=0}^{(w-1)/2}a_{m,n}J_{!}(2n)J_{!}(w-2n)
\]
for $0\leq m\leq(w-1)/2$ where
\[
a_{m,n}=\delta_{m,n}+{w-1-2m \choose 2n-1}+{2m \choose 2n-1}.
\]
Since $a_{m,n}\equiv\delta_{m,n}\pmod{2}$, the matrix $(a_{m,n})_{0\leq m,n\leq w/2}$
is invertible. Thus there are no linear relations among $J_{\mathfrak{D}}(w-2m-1;1,2m)$'s.
\end{proof}
\begin{thm}
There are no linear relations among $J_{\mathfrak{D}}({\rm even};{\rm odd},0)$'s.
\end{thm}

\begin{proof}
Let $w$ be a positive odd integer. By (\ref{eq:odd_case}), 
\[
J_{!}(2m;w-2m,0)=\sum_{n=0}^{(w-1)/2}a_{m,n}J_{!}(2n)J_{!}(w-2n)
\]
for $0\leq m\leq(w-1)/2$ where
\[
a_{m,n}={2m \choose w-2n}+{w-2m \choose 2n}.
\]
Then the matrix $M_{w}=(a_{m,n})_{0\leq m,n\leq(w-1)/2}$ is invertible
since
\[
M_{w}\equiv\left({w-2m \choose 2n}\right)_{0\leq m,n\leq(w-1)/2}\equiv\left(\begin{array}{ccc}
* &  & 1\\
 & \iddots\\
1 &  & 0
\end{array}\right)\pmod{2}.
\]
Thus there are no linear relations among $J_{\mathfrak{D}}(2m;w-2m,0)$'s.
\end{proof}
\begin{rem}
We found following observation concerning the determinant of $M_{w}$:
\[
\det(M_{w})\overset{?}{=}\prod_{\substack{n=1\\
n:{\rm odd}
}
}^{w}c_{n}
\]
where
\[
c_{n}=(-1)^{(n-1)/2}\times\left(\left(\frac{1+\sqrt{5}}{2}\right)^{n}+\left(\frac{1-\sqrt{5}}{2}\right)^{n}\right)\begin{cases}
1 & n\equiv1,5\pmod{6}\\
\frac{1}{4} & n\equiv3\pmod{6}.
\end{cases}
\]
\end{rem}

\begin{thm}
There are no linear relations among $J_{\mathfrak{D}}(2m+1;2n,0)$'s.
\end{thm}

\begin{proof}
Let $w$ be a positive odd integer. By (\ref{eq:odd_case}), 
\[
J_{!}(w-2m;2m,0)=\sum_{n=0}^{(w-1)/2}\left(-{w-2m \choose 2n-2m}-{2m \choose 2n}+\delta_{n,m}\right)J_{!}(2n)J_{!}(w-2n).
\]
for $0\leq m\leq(w-1)/2$. Thus the claim is equivalent to
\[
A_{w}:=\det\left(-{w-2m \choose 2n-2m}-{2m \choose 2n}+\delta_{n,m}\right)_{0\leq m,n\leq(w-1)/2}\neq0.
\]
Define $\text{\ensuremath{\eta}}_{w}:V_{w}^{+}\to V_{w}^{+}$ by
\[
\text{\ensuremath{\eta}}_{w}(p(X,Y))=\frac{1}{2}\left(-p(X,X+Y)-p(X,-X+Y)-p(X-Y,Y)-p(X+Y,Y)+2p(X,Y)\right).
\]
Then
\[
\det\text{\ensuremath{\eta}}_{w}=A_{w}.
\]
Define $\Phi'_{w}:V_{w}^{+}\to V_{w}$ and $\Phi_{w}:V_{w}^{+}\to V_{w}^{+}$
by
\[
\Phi'_{w}(p(X,Y))=p(Y,X+Y)-p(X+Y,X)+p(X,Y)
\]
and
\[
\Phi_{w}(p(X,Y))=\Phi_{w}'(p(X,Y)+p(-X,Y)).
\]
Then
\begin{align*}
2\Phi'_{w}(\text{\ensuremath{\eta}}_{w}(p(X,Y))) & \equiv\left(p(Y,X-2Y)+2p(Y,X+Y)\right)\\
 & \ -3p(X,X+Y)-3p(X+Y,Y)\\
 & \ +\left(-p(X+2Y,X+Y)+p(X-Y,-2X+Y)\right)\\
 & \ +\left(-p(-2X+Y,X)-2p(X+Y,X)\right)\\
 & \ +2p(X,Y)\ \ \pmod{V_{w}^{-}}
\end{align*}
for $p(X,Y)\in V_{w}^{+}$. We define the integral part of $V_{w}^{+}$
by $V_{w,\mathbb{Z}}^{+}:=V_{w}^{+}\cap\mathbb{Z}[X,Y]$. Then for
$p(X,Y)\in V_{w,\mathbb{Z}}^{+}$,
\[
\Phi_{w}\circ\text{\ensuremath{\eta}}_{w}(p(X,Y))\equiv p(X,Y)\pmod{3V_{w,\mathbb{Z}}^{+}}.
\]
Thus
\[
A_{w}\not\equiv0\pmod{3}.
\]
Therefore, the theorem is proved.
\end{proof}
\begin{thm}
There are no linear relations among $J_{\mathfrak{D}}({\rm even};0,{\rm odd})$'s
except for
\[
\sum_{m=0}^{(w-1)/2}J_{\mathfrak{D}}(2m;0,w-2m)=0
\]
where $w$ is a positive odd integer.
\end{thm}

\begin{proof}
Let $w$ be a positive odd integer. By (\ref{eq:odd_case}), 
\[
J_{!}(2m;0,w-2m)=\sum_{n=1}^{(w-1)/2}\left({2m \choose 2n}-{w-2m \choose 2n}\right)J_{!}(2n)J_{!}(w-2n)
\]
for $0\leq m\leq(w-1)/2$. Thus the identity $\sum_{m=0}^{(w-1)/2}J_{\mathfrak{D}}(2m;0,w-2m)=0$
can be checked easily from this, and the claim is equivalent to
\[
\det\left({2m \choose 2n}-{w-2m \choose 2n}\right)_{1\leq m,n\leq(w-1)/2}\neq0.
\]
For $1\leq n\leq(w-1)/2$, since
\[
{2m \choose 2n}-{w-2m \choose 2n}={\frac{w}{2}+x \choose 2n}-{\frac{w}{2}-x \choose 2n}\ \ \ (x:=2m-w/2),
\]
there exists a polynomial $P_{n}$ of degree $n-1$ such that
\[
{2m \choose 2n}-{w-2m \choose 2n}=\left(2m-\frac{w}{2}\right)P_{n}\left(\left(2m-\frac{w}{2}\right)^{2}\right).
\]
Since
\[
xP(x^{2})=\frac{1}{(2n)!}\left(\prod_{j=0}^{2n-1}\left(\frac{w}{2}-j+x\right)-\prod_{j=0}^{2n-1}\left(\frac{w}{2}-j-x\right)\right),
\]
the coefficient of $x^{n-1}$ in $P_{n}$ is
\[
\frac{2}{(2n)!}\sum_{j=0}^{2n-1}\left(\frac{w}{2}-j\right)=\frac{w-2n+1}{(2n-1)!}.
\]
Thus, by the Vandermonde's determinant formula, we have
\begin{align*}
\det\left({2m \choose 2n}-{w-2m \choose 2n}\right)_{1\leq m,n\leq(w-1)/2} & =\det\left(\left(2m-\frac{w}{2}\right)P_{n}\left(\left(2m-\frac{w}{2}\right)^{2}\right)\right)_{1\leq m,n\leq(w-1)/2}\\
 & =\left(\prod_{m=1}^{(w-1)/2}\left(2m-\frac{w}{2}\right)\right)\left(\prod_{1\leq i<j\leq(w-1)/2}\left(\left(2j-\frac{w}{2}\right)^{2}-\left(2i-\frac{w}{2}\right)^{2}\right)\right)\\
 & \ \ \ \ \ \times\left(\prod_{n=1}^{(w-1)/2}\frac{w-2n+1}{(2n-1)!}\right)\\
 & =\left(\frac{w-1}{2}\right)!\times\begin{cases}
-1 & w\equiv5\pmod{8}\\
1 & {\rm otherwise}.
\end{cases}\\
 & \neq0.
\end{align*}
Therefore the theorem is proved.
\end{proof}

\section{Remarks\label{sec:Remarks}}

In this section, we make some comment for the cases not treated in
Theorems \ref{thm:first_main_result} and \ref{thm:second_main_result}.
For the cases $J(1;{\rm odd},{\rm odd})$ and $J({\rm odd};{\rm odd},1)$,
we have following conjectures.
\begin{conjecture}
\label{conj:no_rel_1_odd_odd}There are no linear relations among
$J_{\mathfrak{D}}(1;{\rm odd},{\rm odd})$'s.
\end{conjecture}

\begin{conjecture}
\label{conj:no_rel_odd_odd_1}There are no linear relations among
$J_{\mathfrak{D}}({\rm odd};{\rm odd},1)$'s.
\end{conjecture}

Let $w$ be a positive even integer. By (\ref{eq:odd_case}),
\[
J_{!}(1;w-2m-1,2m+1)=\sum_{n=0}^{w/2-1}\left(-\delta_{m,n}+{w-2m-1 \choose 2n}-{2m+1 \choose 2n}\right)J_{!}(w-2n)J_{!}(2n+1)
\]
and
\[
J_{!}(2m+1;w-2m-1,1)=\sum_{n=0}^{w/2-1}\left(-\delta_{m,n}+{w-2m-1 \choose 2n-2m}-{2m+1 \choose 2n}\right)J_{!}(w-2n)J_{!}(2n+1)
\]
by Lemma \ref{lem:vanish_condition} ($n=w/2$ terms vanish). Thus
Conjectures \ref{conj:no_rel_1_odd_odd} and \ref{conj:no_rel_odd_odd_1}
are equivalent to
\[
\det\left(-\delta_{m,n}+{w-2m-1 \choose 2n}+{2m+1 \choose 2n}\right)_{0\leq m,n\leq w/2-1}\neq0
\]
and
\[
\det\left(-\delta_{m,n}+{w-2m-1 \choose 2n-2m}-{2m+1 \choose 2n}\right)_{0\leq m,n\leq w/2-1}\neq0,
\]
respectively. Numerical calculation\footnote{We check this up to $w\leq200$.}
suggests
\[
{\rm sgn}\left(\det\left(-\delta_{m,n}+{w-2m-1 \choose 2n}-{2m+1 \choose 2n}\right)_{0\leq m,n\leq w/2-1}\right)=\begin{cases}
1 & w/2\equiv0,2,5\pmod{6}\\
-1 & w/2\equiv1,3,4\pmod{6}.
\end{cases}
\]
\[
{\rm sgn}\left(\det\left(-\delta_{m,n}+{w-2m-1 \choose 2n-2m}-{2m+1 \choose 2n}\right)_{0\leq m,n\leq w/2-1}\right)=(-1)^{w/2}.
\]
Thus one might be able to prove Conjectures \ref{conj:no_rel_1_odd_odd}
and \ref{conj:no_rel_odd_odd_1} by giving a good approximation of
the value of these determinants.

For other cases, numeric calculation suggests that the dimensions
of the linear relations among $J_{\mathfrak{D}}({\rm even};{\rm even},1)$'s,
$J_{\mathfrak{D}}(1;{\rm odd},{\rm even})$'s, $J_{\mathfrak{D}}({\rm odd};1,{\rm even})$'s,
$J_{\mathfrak{D}}({\rm odd};{\rm even},1)$'s, and $J_{\mathfrak{D}}({\rm even};{\rm odd},1)$'s
are given by $1$, $\dim S_{k}({\rm SL}_{2}(\mathbb{Z}))$, $\dim S_{k}({\rm SL}_{2}(\mathbb{Z}))$,
$\dim M_{k}({\rm SL}_{2}(\mathbb{Z}))$, and $\dim M_{k}({\rm SL}_{2}(\mathbb{Z}))$,
respectively where $k$ is a weight. However, the author could not
find any good description of these relations.

\section*{Acknowledgements}

This work was supported by JSPS KAKENHI Grant Numbers, JP18J00982
and JP18K13392. The author would like to thank Koji Tasaka and Nobuo
Sato for useful comments on the draft of this paper.

\appendix

\section{\label{sec:Period-polynomials}Period polynomials for ${\rm SL}_{2}(\mathbb{Z})$,
$\Gamma_{0}(2)$, $\Gamma_{A}$, $\Gamma_{B}$}

\subsection{Popa and Pasol's formulation of period polynomials}

In this section, we state the formulation of period polynomials stated
in \cite{Pasol_Popa_periodpolynomial}. Fix a non-negative integer
$w$ and a finite index congruence subgroup $\Gamma\subset{\rm SL}_{2}(\mathbb{Z})$.
We assume that $\epsilon\Gamma\epsilon=\Gamma$. We denote by $V_{w}$
the space of degree $w$ homogeneous polynomial of $X$ and $Y$.
We put $\hat{V}_{w}=\frac{1}{XY}V_{w+2}$. We denote by $\hat{V}_{w}^{\Gamma}$
the space of maps $P:\Gamma\backslash{\rm SL}_{2}(\mathbb{Z})\to\hat{V}_{w}$
such that $P\left(\left(\begin{smallmatrix}-1 & 0\\
0 & -1
\end{smallmatrix}\right)C\right)=(-1)^{w}P(C)$ for all $C\in\Gamma\backslash{\rm SL}_{2}(\mathbb{Z})$. We define
the right action of ${\rm GL}_{2}(\mathbb{Z})$ to $\hat{V}_{w}^{\Gamma}$
by $\left.P\right|_{\gamma}(C)=\left.P(\epsilon^{m}C\gamma^{-1})\right|_{\gamma}$
where $m=0$ if $\det(\gamma)=1$ and $m=1$ if $\det(\gamma)=-1$.
This action does not preserve $\hat{V}_{w}^{\Gamma}$, but we can
define the following spaces
\begin{align*}
\hat{\mathbb{\mathcal{W}}}_{w}^{\Gamma} & =\{P\in\hat{V}_{w}^{\Gamma}:\left.P\right|_{1+S}=\left.P\right|_{1+U+U^{2}}=0\},\\
\mathcal{W}_{w}^{\Gamma} & =\{P\in\hat{\mathbb{\mathcal{W}}}_{w}^{\Gamma}:P(C)\in V_{w}\ \text{for all }C\in\Gamma\backslash{\rm SL}_{2}(\mathbb{Z})\}.
\end{align*}
Note that the action of $\epsilon$ preserve $\hat{\mathbb{\mathcal{W}}}_{w}^{\Gamma}$
and $\mathcal{W}_{w}^{\Gamma}$. Thus we can get the decomposition
$\hat{\mathbb{\mathcal{W}}}_{w}^{\Gamma}=\hat{\mathbb{\mathcal{W}}}_{w}^{\Gamma,+}\oplus\hat{\mathbb{\mathcal{W}}}_{w}^{\Gamma,-}$
and $\mathcal{W}_{w}^{\Gamma}=\mathcal{W}_{w}^{\Gamma,+}\oplus\mathcal{W}_{w}^{\Gamma,-}$
where 
\[
\hat{\mathbb{\mathcal{W}}}_{w}^{\Gamma,\pm}=\{P\in\hat{\mathbb{\mathcal{W}}}_{w}^{\Gamma}:\left.P\right|_{\epsilon}=\pm P\},\ \mathcal{W}_{w}^{\Gamma,\pm}=\{P\in\mathcal{W}_{w}^{\Gamma}:\left.P\right|_{\epsilon}=\pm P\}.
\]
For $f\in M_{w+2}(\Gamma)$ and $C\in\Gamma\backslash{\rm SL}_{2}(\mathbb{Z})$,
define $\left.f\right|_{C}\in M_{w+2}(C^{-1}\Gamma C)$ by $\left.f\right|_{C}(z)=(cz+d)^{-w-2}f(\frac{az+b}{cz+d})$
where $\left(\begin{smallmatrix}a & b\\
c & d
\end{smallmatrix}\right)\in C$. Furthermore, define $\hat{\rho}_{f}\in\hat{V}_{w}^{\Gamma}$ by
$\hat{\rho}_{f}(C)=P_{\left.f\right|_{C}}$. Then we can show $\hat{\rho}_{f}\in\hat{\mathbb{\mathcal{W}}}_{w}^{\Gamma}$
for $f\in M_{w+2}(\Gamma)$ and $\hat{\rho}_{f}\in\mathcal{W}_{w}^{\Gamma}$
for $f\in S_{w+2}(\Gamma)$. We put $\hat{\rho}_{f}^{\pm}=\frac{1}{2}(\hat{\rho}_{f}\pm\left.\hat{\rho}_{f}\right|_{\epsilon})$
and $\hat{\rho}^{\pm}(f)=\hat{\rho}_{f}^{\pm}$.
\begin{thm}[{\cite[Theorem 2.1]{Pasol_Popa_periodpolynomial}, Popa--Pasol, Eichler--Shimura}]
\label{thm:Popa_pasol_cuspform}Let $C_{w}^{\Gamma,\pm}$ be the
subspace of $\mathcal{W}_{w}^{\Gamma,\pm}$ defined in \cite{Pasol_Popa_periodpolynomial}.
Then the composite map
\[
S_{w+2}(\Gamma)\xrightarrow{\hat{\rho}^{\pm}}\mathcal{W}_{w}^{\Gamma,\pm}\to\mathcal{W}_{w}^{\Gamma}/C_{w}^{\Gamma,\pm}
\]
is an isomorphism.
\end{thm}

Here we do not explain the definition of $C_{w}$ because we only
use the injectivity of $S_{w+2}(\Gamma)\xrightarrow{\hat{\rho}_{f}^{\pm}}\mathcal{W}_{w}^{\Gamma,\pm}$.
\begin{thm}[{\cite[Proposition 8.4]{Pasol_Popa_periodpolynomial}}]
\label{thm:Popa_pasol_1}The spaces $M_{w+2}(\Gamma)$ and $\hat{\mathbb{\mathcal{W}}}_{w}^{\Gamma,\pm}$
have same dimensions. If the extended Peterson Scalar product on $M_{w+2}(\Gamma)$
defined in \cite{Pasol_Popa_innerproduct} is nondegenerate then the
maps $\hat{\rho}^{\pm}:M_{k}(\Gamma)\to\hat{\mathbb{\mathcal{W}}}_{w}^{\Gamma,\pm},f\mapsto\hat{\rho}_{f}^{\pm}$
are isomorphisms.
\end{thm}

\begin{thm}[{\cite[Theorem 4.4]{Pasol_Popa_innerproduct}}]
\label{thm:Popa_pasol_2}For $k>2$ and $\Gamma\in\{\Gamma_{1}(N),\Gamma_{0}(N)\}$,
the extended Peterson Scalar product on $M_{k}(\Gamma)$ is nondegenerate.
Furthermore, for $k=2$ and $\Gamma\in\{\Gamma_{1}(p),\Gamma_{0}(p)\}$
with $p$ prime, the extended Peterson Scalar product on $M_{k}(\Gamma)$
is also nondegenerate.
\end{thm}

\subsection{Period polynomials for ${\rm SL}_{2}(\mathbb{Z})$, $\Gamma_{0}(2)$,
$\Gamma_{A}$, $\Gamma_{B}$}

In this subsection, we show the following.
\begin{prop}
\label{prop:period_polynomials}The even period maps 
\[
M_{k}(\Gamma)\to W_{w}^{+,\Gamma}\ ;\ f\mapsto P_{f}^{+}
\]
for $\Gamma={\rm SL}_{2}(\mathbb{Z})$, $\Gamma_{0}(2)$, and $\Gamma_{A}$,
and the odd period maps
\[
\tilde{S}_{k}(\Gamma)\to W_{w}^{-,\Gamma}\ ;\ f\mapsto P_{f}^{-}
\]
for $\Gamma={\rm SL}_{2}(\mathbb{Z})$, $\Gamma_{A}$, and $\Gamma_{B}$
are injective. Furthermore,
\[
W_{w}^{+}=\{p\in V_{w}:\left.p\right|_{1-\epsilon}=\left.p\right|_{1+S}=\left.p\right|_{1+U+U^{2}}=0\},
\]
\[
W_{w}^{+,\Gamma_{0}(2)}=\{p\in V_{w}:\left.p\right|_{1-\epsilon}=\left.p\right|_{(1-T)(1+M)}=0\},
\]
\[
W_{w}^{+,\Gamma_{A}}=\{P\in V_{w}:\left.p\right|_{1-\epsilon}=\left.p\right|_{1+U+U^{2}}=\left.p\right|_{S+SU+SU^{2}}=0\},
\]
\[
W_{w}^{-}=\{p\in V_{w}:\left.p\right|_{1+\epsilon}=\left.p\right|_{1+S}=\left.p\right|_{1+U+U^{2}}=0\},
\]
\[
W_{w}^{-,\Gamma_{A}}=\{p\in V_{w}:\left.p\right|_{1+\epsilon}=\left.p\right|_{1+U+U^{2}}=\left.p\right|_{S+SU+SU^{2}}=0\},
\]
\[
W_{w}^{-,\Gamma_{B}}=\{p\in V_{w}:\left.p\right|_{1+\epsilon}=\left.p\right|_{1+S}=0\}.
\]
\end{prop}

The case $\Gamma={\rm SL}_{2}(\mathbb{Z})$ is obvious and well-known.
So let us check the other cases.
\begin{lem}
\label{lem:The-even-period-MkGamma02}The even period map $M_{k}(\Gamma_{0}(2))\to W_{w}^{+,\Gamma_{0}(2)}$
is injective (and thus an isomorphism) and 
\[
W_{w}^{+,\Gamma_{0}(2)}=\{p\in V_{w}:\left.p\right|_{1-\epsilon}=\left.p\right|_{(1-T)(1+M)}=0\}.
\]
\end{lem}

\begin{proof}
By Theorems \ref{thm:Popa_pasol_1} and \ref{thm:Popa_pasol_2}, $\hat{\rho}^{\pm}:M_{k}(\Gamma_{0}(2))\to\hat{\mathbb{\mathcal{W}}}_{w}^{\Gamma_{0}(2),\pm}$
are isomorphisms. We denote by $C$, $C'$, and $C''$ the cosets
represented by $1$, $U$, $U^{2}$ respectively. Then by definition,$\hat{\mathbb{\mathcal{W}}}_{w}^{\Gamma_{0}(2),\pm}$
is the space of maps $P$ from $\{C,C',C''\}$ to $\hat{V}_{w}$ such
that 
\[
0=P(C)\mid_{S}+P(C')=\left.P(C'')\right|_{1+S}=P(C)+\left.P(C')\right|_{U^{2}}+\left.P(C'')\right|_{U}=\left.P(C)\right|_{1-\epsilon}=\left.P(C'')\right|_{1-\epsilon}.
\]
Thus $P\in\hat{\mathbb{\mathcal{W}}}_{w}^{\Gamma_{0}(2),\pm}$ is
determined by $P(C)$ since $P(C')=-P(C)\mid_{S}$ and $P(C'')=P(C)\mid_{SU-U^{2}}$.
Thus the even period map $M_{k}(\Gamma_{0}(2))\to W_{w}^{+,\Gamma_{0}(2)}$
is injective. Let $P\in\hat{\mathbb{\mathcal{W}}}_{w}^{\Gamma_{0}(2),\pm}$.
By $P(C'')=P(C)\mid_{-U^{2}+SU}$ and $P(C'')\mid_{1+S}=0$, 
\begin{align*}
0 & =P(C)\mid_{(SU-U^{2})(1+S)}=P(C)\mid_{(1-T)SU(1+S)}=P(C)\mid_{(1-T)(1+M)SUS}.
\end{align*}
Thus $P(C)\mid_{(1-T)(1+M)}=0$. Therefore, 
\[
W_{w}^{+,\Gamma_{0}(2)}\subset\{p\in V_{w}:\left.p\right|_{1-\epsilon}=\left.p\right|_{(1-T)(1+M)}=0\}.
\]
Let $p\in\{p\in V_{w}:\left.p\right|_{1-\epsilon}=\left.p\right|_{(1-T)(1+M)}=0\}$.
Define $P:\{C,C',C''\}\to V_{w}$ by $P(C)=p$, $P(C')=-\left.p\right|_{S}$,
and $P(C'')=\left.p\right|_{SU-U^{2}}$. Then $P\in\hat{\mathbb{\mathcal{W}}}_{w}^{\Gamma_{0}(2),\pm}$.
Thus
\[
\{p\in V_{w}:\left.p\right|_{1-\epsilon}=\left.p\right|_{(1-T)(1+M)}=0\}\subset W_{w}^{+,\Gamma_{0}(2)}.
\]
Thus the lemma is proved.
\end{proof}
\begin{lem}
The odd period map $\tilde{S}_{k}(\Gamma_{B})\to W_{w}^{-,\Gamma_{B}}$
is injective (and thus an isomorphism) and 
\[
W_{w}^{-,\Gamma_{B}}=\{p\in V_{w}:\left.p\right|_{1+\epsilon}=\left.p\right|_{1+S}=0\}.
\]
\end{lem}

\begin{proof}
Let $C$, $C'$and $C''$ be the same as in the proof of Lemma \ref{lem:The-even-period-MkGamma02}.
We remark that 
\[
M_{k}(\Gamma_{B})=\{\left.f\right|_{C''}:f\in M_{k}(\Gamma_{0}(2))\}.
\]
Note that
\[
W_{w}^{-,\Gamma_{B}}=\{P(C''):P\in\hat{\mathbb{\mathcal{W}}}_{w}^{\Gamma_{0}(2),-}\}\cap V_{w}.
\]
Let $P\in\hat{\mathbb{\mathcal{W}}}_{w}^{\Gamma_{0}(2),-}$ and $p=P(C'')$.
Then
\[
\left.P(C)\right|_{T-1}=\left.p\right|_{U}.
\]
Since $\left.P(C)\right|_{1+\epsilon}=0$, $P(C)$ is uniquely determined
by $p$. Thus $\tilde{S}_{k}(\Gamma_{B})\to W_{w}^{-,\Gamma_{B}}$
is injective. The inclusion
\[
W_{w}^{-,\Gamma_{B}}\subset\{p\in V_{w}:\left.p\right|_{1+\epsilon}=\left.p\right|_{1+S}=0\}
\]
is obvious from the definition.

Let $p\in\{p\in V_{w}:\left.p\right|_{1+\epsilon}=\left.p\right|_{1+S}=0\}$.
Then there is unique $q\in\frac{X}{Y}V_{w}$ such that
\[
\left.q\right|_{1-T}=\left.p\right|_{U}.
\]
Then
\[
\left.q\right|_{(1+\epsilon)(1-T)}=\left.p\right|_{U-\epsilon SU}=0.
\]
Thus $\left.q\right|_{1+\epsilon}\in\mathbb{Q}Y^{w}$. Since $q\in\frac{X}{Y}V_{w}$,
$\left.q\right|_{1+\epsilon}=0$. Now, if we define $P:\{C,C',C''\}\to\hat{V}_{w}$
by
\[
P(C)=q,\ P(C')=-\left.q\right|_{S},\ P(C'')=p,
\]
then $P\in\hat{\mathbb{\mathcal{W}}}_{w}^{\Gamma_{0}(2),-}$. Thus
\[
\{p\in V_{w}:\left.p\right|_{1+\epsilon}=\left.p\right|_{1+S}=0\}\subset W_{w}^{-,\Gamma_{B}}.
\]
Hence the lemma is proved.
\end{proof}
\begin{lem}
Then even (resp. odd) period map $M_{k}(\Gamma_{A})\to W_{w}^{+,\Gamma_{A}}$
(resp. $S_{k}(\Gamma_{A})\to W_{w}^{-,\Gamma_{A}}$) is injective
(and thus an isomorphism) and 
\[
W_{w}^{\pm,\Gamma_{A}}=\{p\in V_{w}^{\pm}:\left.p\right|_{1+U+U^{2}}=\left.p\right|_{S+SU+SU^{2}}=0\}.
\]
\end{lem}

\begin{proof}
First, let us prove the injectivity of $\hat{\rho}^{\pm}:M_{k}(\Gamma_{A})\to\hat{\mathbb{\mathcal{W}}}_{w}^{\Gamma_{A},\pm}$.
Let $f\in M_{k}(\Gamma_{A})$ and assume that $\hat{\rho}^{\pm}(f)=0$.
Then $f$ can be written as $f=f^{'}+f''$ where $f'\in M_{k}({\rm SL}_{2}(\mathbb{Z}))$
and $f''\in\sqrt{\Delta}M_{k-6}({\rm SL}_{2}(\mathbb{Z}))$. Then
$\hat{\rho}^{\pm}(f')=\hat{\rho}^{\pm}(f'')=0$ since $\hat{\rho}^{\pm}(f)=0$,
$f'=\frac{1}{2}\left.f\right|_{1+S}$, and $f''=\frac{1}{2}\left.f\right|_{1-S}$.
Thus $f'=0$ by the injectivity of the period map for ${\rm SL}_{2}(\mathbb{Z})$,
and $f''=0$ by the injectivity of $\hat{\rho}^{\pm}:S_{k}(\Gamma_{A})\to\hat{\mathbb{\mathcal{W}}}_{w}^{\Gamma_{A},\pm}$
(Theorem \ref{thm:Popa_pasol_cuspform}). Thus $\hat{\rho}^{\pm}:M_{k}(\Gamma_{A})\to\hat{\mathbb{\mathcal{W}}}_{w}^{\Gamma_{A},\pm}$
is injective and therefore an isomorphism. The claim of the lemma
easily follows from this isomorphism.
\end{proof}
\bibliographystyle{plain}
\bibliography{reference}

\end{document}